\documentclass[12pt,leqno]{amsart}
\usepackage[latin9]{inputenc}
\usepackage{verbatim}
\usepackage{amsthm}
\usepackage{amstext}
\usepackage{amssymb}
\usepackage{amscd}
\usepackage{url}

\makeatletter
\numberwithin{equation}{section}
\numberwithin{figure}{section}

\usepackage{amsthm}\usepackage{amscd}\usepackage{bm}
\theoremstyle{plain}
\newtheorem{theorem}{Theorem}[section]
\newtheorem{lemma}[theorem]{Lemma}
\newtheorem{conjecture}[theorem]{Conjecture}
\newtheorem{corollary}[theorem]{Corollary}
\newtheorem{proposition}[theorem]{Proposition}
\newtheorem{claim}[theorem]{Claim}

\theoremstyle{definition}

\newtheorem{remark}[theorem]{Remark}
\newtheorem{example}[theorem]{Example}

\newtheorem{problem}[theorem]{Problem}

\newcommand{\const}{\mathrm{const}}
\newcommand{\norm}[1]{\left|\!\left|#1\right|\!\right|}
\newcommand{\nnorm}[1]{\left|\!\left|\!\left|#1\right|\!\right|\!\right|}
\newcommand{\ad}{\mathrm{ad}}
\newcommand{\supp}{\mathrm{supp}}
\newcommand{\vol}{\mathrm{vol}}
\newcommand{\Diam}{\mathrm{Diam}}
\newcommand{\moduli}{\mathcal{M}}
\newcommand{\widim}{\mathrm{Widim}}
\newcommand{\dist}{\mathrm{dist}}

\newcommand{\mmdim}{\mathrm{dim}_{\mathrm{M}}}
\newcommand{\tnorm}[1]{\left|\!\left|#1\right|\!\right|_{\mathrm{Tau}}}

\begin{document}

\title[Large dynamics of Yang--Mills theory]
{Large dynamics of Yang--Mills theory: mean dimension formula} 

\author[M. Tsukamoto]{Masaki Tsukamoto}

\subjclass[2010]{58D27, 53C07, 37B99}

\keywords{Yang--Mills gauge theory, mean dimension, metric mean dimension}

\thanks{This paper was supported by Grant-in-Aid for Young Scientists (B) 
25870334 from JSPS}

\date{\today}

\maketitle

\begin{abstract}
This paper studies the Yang--Mills ASD equation over the cylinder as a non-linear evolution equation.
We consider a dynamical system consisting of bounded orbits of this evolution equation.
This system contains many chaotic orbits, and moreover it becomes an infinite dimensional and 
infinite entropy system.
We study the mean dimension of this huge dynamical system.
Mean dimension is a topological invariant of dynamical systems introduced by Gromov.
We prove the exact formula of the mean dimension by 
developing a new technique based on the metric mean dimension theory of Lindenstrauss--Weiss.
\end{abstract}

\section{Introduction}  \label{section: introduction}

\subsection{Main result}  \label{subsection: main result}

This paper explores a large chaotic dynamics of \textbf{Yang--Mills gauge theory}.
Yang--Mills theory is the study of special connections (Yang--Mills connections, ASD connections and 
its perturbations) on principal fiber bundles over manifolds.
Its origin is quantum physics, and it has been intensively
studied in differential/algebraic geometry, low-dimensional topology and representation theory.
Many astonishing results have been obtained for more than 30 years.
But its \textit{dynamical} aspect has been largely neglected.
The purpose of the paper is to reveal a new rich dynamical structure of Yang--Mills theory.

Traditionally most researchers in Yang--Mills theory have been interested in 
highly concentrated special connections called \textbf{instantons}.
Probably this is a reason why dynamical aspect of the theory has not attract their attentions for a long time.
When we look at only concentrated solutions, we don't need a dynamical point of view.
Dynamics appears only when we are interested in a very long term phenomena.
For example, calculating geodesics on Riemannian manifolds is the simplest problem in calculus of variations.
But when we look at very long geodesics (i.e. geodesic flow), we face a complicated dynamical problem.

To explain our viewpoint more concretely,
we recall a familiar picture of instanton Floer homology
(Floer \cite{Floer} and Donaldson \cite{Donaldson}).
Let $Y$ be a closed oriented Riemannian 3-manifold, and we consider the cylinder 
$\mathbb{R}\times Y$ with the product metric. We denote its $\mathbb{R}$-coordinate by $t$.
Let $E$ be a principal $SU(2)$ bundle over $\mathbb{R}\times Y$.
A connection $A$ on $E$ is said to be \textbf{anti-self-dual} (ASD) if its curvature $F_A$ is 
anti-self-dual with respect to the Hodge star operation:
\[ * F_A = -F_A.\]
It is a crucial point in Floer theory that this equation can be expressed as a non-linear \textit{evolution equation}. 
Suppose $A$ is expressed in the temporal gauge, i.e. it has no $dt$-part.
Then the ASD equation becomes
\begin{equation}  \label{eq: ASD equation in the temporal gauge}
 \frac{\partial A(t)}{\partial t} = -*_3 F(A(t)), 
\end{equation}
where $A(t)$ is the restriction of $A$ to the section $\{t\}\times Y$. 
Fixed points of the equation (\ref{eq: ASD equation in the temporal gauge}) are flat connections, and 
connecting orbits between fixed points correspond to instantons.
Floer homology is constructed by using these objects.
Generators of Floer chain complex are flat connections, and the differentials involve
instanton counting.
Therefore we can say that Floer homology uses some dynamics of the evolution equation
(\ref{eq: ASD equation in the temporal gauge}).

But the equation (\ref{eq: ASD equation in the temporal gauge}) also contains more 
complicated dynamical objects other than fixed points and connecting orbits.
Firstly the equation (\ref{eq: ASD equation in the temporal gauge}) admits many periodic orbits. 
Periodic points of period $T>0$ correspond to instantons over $(\mathbb{R}/T\mathbb{Z})\times Y$, and
a lot of such solutions can be constructed by using the gluing theorem of Taubes \cite{Taubes}.
Secondly, and more importantly, 
the above evolution equation contains many chaotic orbits similar to ones in the \textbf{Bernoulli shift} $\{0,1\}^{\mathbb{Z}}$.
This can be shown by using \textit{infinite gluing technique} \cite{Tsukamoto-Nagoya-gluing, Tsukamoto-MPAG} as follows.
Pick up two sufficiently concentrated instantons $A_0$ and $A_1$ over the Euclidean space $\mathbb{R}^4$.
We consider the gluing of infinitely many copies of $A_0$ and $A_1$ over $\mathbb{R}\times Y$.
Take a point $x=(x_n)_n$ in the Bernoulli shift $\{0,1\}^{\mathbb{Z}}$.
For each $n\in \mathbb{Z}$ we glue $A_0$ or $A_1$ in a neighborhood of $\{t=n\}$ 
depending on whether $x_n=0$ or $x_n=1$.
Then, in a rough expression, the resulting ASD connection $A_x$ looks like 
\[ A_x = \cdots \sharp A_{x_{-1}}\sharp A_{x_0} \sharp A_{x_1} \sharp \cdots.\]
The dynamical behavior of $A_x$ imitates
that of the point $x$ in the Bernoulli shift, and it is generically chaotic.

Indeed the dynamics of (\ref{eq: ASD equation in the temporal gauge}) is much more complicated than 
the Bernoulli shift. 
Suppose $A_0$ and $A_1$ admit non-trivial deformation.
Then each $A_{x_n}$ can be deformed. 
So the ASD connection $A_x$ has infinitely many deformation parameters.
This means that the equation (\ref{eq: ASD equation in the temporal gauge}) contains a dynamics 
like $[0,1]^{\mathbb{Z}}$ (the shift action on the \textbf{Hilbert cube}).
$[0,1]^{\mathbb{Z}}$ is an \textit{infinite dimensional} dynamical system of \textit{infinite topological entropy}. 
So this is much larger than the Bernoulli shift.

We have explained that the ASD equation (\ref{eq: ASD equation in the temporal gauge}) 
contains a huge dynamics. 
The purpose of the paper is to develop this unexplored aspect of gauge theory. 
One motivation of this study comes from the work of Gromov \cite{Gromov}.
He introduced a new topological invariant of dynamical systems called \textbf{mean dimension}.
This provides a non-trivial information for infinite dimensional and infinite entropy systems.
For example the $\mathbb{Z}$-action on the Hilbert cube $[0,1]^{\mathbb{Z}}$ has mean dimension $1$.
Mean dimension has been attracting 
researchers in several areas such as topological dynamics 
\cite{Lindenstrauss--Weiss, Lindenstrauss, Gutman, Gutman 2, Lindenstrauss--Tsukamoto, Gutman--Tsukamoto}, 
function theory \cite{Costa, Matsuo--Tsukamoto Brody curve, Tsukamoto-nagoya} 
and operator algebra \cite{Li--Liang, Elliott--Niu}.
We review the definition of mean dimension in Section \ref{subsection: review of mean dimension}.

While the idea of mean dimension is related to various subjects,
Gromov's original motivation is geometric. 
When we study geometric PDE (holomorphic/harmonic maps, complex/minimal subvarieties, etc.)
in a non-compact manifold without any asymptotic boundary condition, we often encounter 
a very large dynamical system (as we have seen above).
Gromov proposed the study of such large dynamical systems from the viewpoint of mean dimension.
Very little has been known in this direction yet.
But here we report one progress of this program in the case of Yang--Mills theory:
We get the exact formula of the mean dimension.
Probably our method can be also applied to other equations.
We will discuss this point again in the end of this subsection.

From now on we concentrate on the simplest case: the 3-manifold $Y$ is the sphere 
$S^3 = \{x_1^2+x_2^2+x_3^2+x_4^2=1\}$ with the standard metric.
Set $X := \mathbb{R}\times S^3$.
The important point is that $X$ (endowed with the product metric) 
is an anti-self-dual manifold  
with a uniformly positive scalar curvature.
Here the anti-self-duality means that the Weyl conformal curvature of $X$ is ASD.
This metrical condition will be used via a certain Weitzenb\"{o}ck formula.
Let $E = X\times SU(2)$ be the product principal $SU(2)$ bundle.
All principal $SU(2)$ bundles over $X$ are isomorphic to the product bundle $E$.
Let $A$ be a connection on $E$.
Its curvature $F_A$ is a 2-form valued in the adjoint bundle $\ad E = X\times su(2)$.
Hence for each point $p\in X$ we can identify $(F_A)_p$ as a linear map 
\begin{equation*} 
 (F_A)_p: \Lambda^2(T_pX)\to su(2).
\end{equation*}
Let $|(F_A)_p|_{\mathrm{op}}$ be its operator norm, and set 
$\norm{F_A}_{\mathrm{op}} := \sup_{p\in X}|(F_A)_p|_{\mathrm{op}}$.

Let $d$ be a non-negative real number.
We define $\moduli_d$ as the space of the gauge equivalence classes of ASD connections $A$ on $E$ satisfying 
\begin{equation} \label{eq: Brody condition}
  \norm{F_A}_{\mathrm{op}} \leq d.
\end{equation} 
This condition (\ref{eq: Brody condition}) means that we consider only 
\textit{bounded} orbits of the evolution equation (\ref{eq: ASD equation in the temporal gauge}).
The space $\moduli_d$ is endowed with the topology of $C^\infty$ convergence over compact subsets:
The sequence $[A_n]$ in $\moduli_d$ converges to $[A]$ if and only if there exist gauge transformations 
$g_n$ satisfying $g_n(A_n)\to A$ in $C^\infty$ over every compact subset of $X$.
The space $\moduli_d$ is compact and metrizable by the Uhlenbeck compactness 
(Uhlenbeck \cite{Uhlenbeck}, Wehrheim \cite{Wehrheim}).

We introduce a continuous action of $\mathbb{R}$ on $\moduli_d$.
This corresponds to the natural time-shift $A(t)\mapsto A(t+s)$ in the 
evolution equation (\ref{eq: ASD equation in the temporal gauge}).
$\mathbb{R}$ acts on $X = \mathbb{R}\times S^3$ by the shift on the $\mathbb{R}$-factor
: $\mathbb{R}\times X\to X$, $(s,(t,\theta))\mapsto (t+s,\theta)$.
This lifts to the action on $E = X\times SU(2)$ by $\mathbb{R}\times E\to E$, 
$(s,(t,\theta, u))\mapsto (t+s,\theta,u)$.
Then $\mathbb{R}$ acts on $\moduli_d$ by 
\begin{equation} \label{eq: action on the moduli}
 \mathbb{R}\times \moduli_d \to \moduli_d, \quad (s,[A])\mapsto [s^*(A)],
\end{equation}
where $s^*(A)$ is the pull-back of $A$ by $s:E\to E$.
We study the dynamics of this action.
This means that we are interested in the asymptotic behavior (as $t\to \pm \infty$) of bounded orbits 
of the evolution equation (\ref{eq: ASD equation in the temporal gauge}).

It is known that $\moduli_d$ for $d<1$ is the one-point space consisting only of the flat connection
(Tsukamoto \cite{Tsukamoto sharp lower bound}).
So this is uninteresting.
But when $d>1$, $\moduli_d$ becomes an infinite dimensional and 
infinite topological entropy system (Matsuo--Tsukamoto \cite{Matsuo--Tsukamoto 2}).
So this is a relevant object of mean dimension theory.
We denote the mean dimension of the action (\ref{eq: action on the moduli})
by $\dim(\moduli_d:\mathbb{R})$.
The mean dimension $\dim(\moduli_d:\mathbb{R})$ is a non-negative real number.
Its rough intuitive meaning is as follows.
Suppose we try to store on computer the orbits of $\moduli_d$ over the time $-T<t< T$ 
up to an error $\varepsilon>0$.
How many memory (/bit) do we need?
It can be estimated by the mean dimension (more precisely \textbf{metric mean dimension}): 
We need at least 
\[ |\log_2 \varepsilon|\, (2T) \dim(\moduli_d:\mathbb{R}) + o(T) \quad (T\to +\infty).\]
This is one expression of a fundamental theorem of Lindenstrauss--Weiss \cite{Lindenstrauss--Weiss}.
See Theorem \ref{thm: metric mean dimension} and discussions around it for more precise explanations.

Our main result is the formula of the mean dimension $\dim(\moduli_d:\mathbb{R})$.
Our formula involves an \textbf{energy density} $\rho(d)$ introduced by Matsuo--Tsukamoto 
\cite{Matsuo--Tsukamoto}.
For $[A]\in \moduli_d$ we define the energy density $\rho(A)$ by 
\begin{equation} \label{eq: energy density}
 \rho(A) := \lim_{T\to +\infty}\left(\frac{1}{8\pi^2 T}\sup_{t\in \mathbb{R}}
  \int_{(t,t+T)\times S^3} |F_A|^2 d\vol\right).
\end{equation}
This limit always exists (Section \ref{subsection: energy density}).
We denote by $\rho(d)$ the supremum of $\rho(A)$ over $[A]\in \moduli_d$.
The energy density $\rho(d)$ is always non-negative and finite.
It is positive for $d>1$ and goes to infinity as $d\to +\infty$
(\cite{Matsuo--Tsukamoto 2}).

The main task of the paper is to prove the upper bound estimate on the mean dimension:
\begin{theorem} \label{thm: main theorem}
\[ \dim(\moduli_d:\mathbb{R})\leq 8\rho(d).\]
\end{theorem}
The lower bound on the mean dimension was already proved by 
Matsuo--Tsukamoto \cite[Theorem 1.2]{Matsuo--Tsukamoto 2}.
Let $\mathcal{D}\subset [0,+\infty)$ be the set of left-discontinuous points of 
the function $\rho(d)$:
\[ \mathcal{D} = \{d\in [0,+\infty)|\, \lim_{\varepsilon\to +0}\rho(d-\varepsilon) \neq \rho(d)\}.\]
This set is at most countable because $\rho(d)$ is a monotone function.
From \cite[Theorem 1.2]{Matsuo--Tsukamoto 2} (see also Remark \ref{remark: local mean dimension} below)
\begin{equation} \label{eq: lower bound on the mean dimension}
   \dim(\moduli_d:\mathbb{R})\geq 8\rho(d),\quad (d\in [0,+\infty)\setminus \mathcal{D}).
\end{equation}
Therefore we get:
\begin{corollary} \label{cor: main corollary}
For $d\in [0,+\infty)\setminus \mathcal{D}$,
\[ \dim(\moduli_d:\mathbb{R}) = 8\rho(d).\]
\end{corollary}
Since $\mathcal{D}$ is at most countable, we get the formula of the mean dimension 
$\dim(\moduli_d:\mathbb{R})$ for almost every $d\geq 0$.
This formula can be seen as a dynamical analogue of the pioneering work of 
Atiyah--Hitchin--Singer \cite[Theorem 6.1]{Atiyah--Hitchin--Singer}.
Here we briefly recall their result.
Let $A$ be an irreducible ASD connection on a principal $SU(2)$ bundle $P$ over 
a \textit{compact} ASD 4-manifold $M$ of positive scalar curvature.
Atiyah--Hitchin--Singer calculated the number of the deformation parameters of $A$ by using the 
Atiyah--Singer index theorem. 
The answer is given by  
\[ 8c_2(P) - 3(1-b_1(M)) \> \text{ where } c_2(P) = \frac{1}{8\pi^2}\int_M |F_A|^2d\vol.\]
Corollary \ref{cor: main corollary} is clearly analogous to this dimension formula.
The energy density (\ref{eq: energy density}) is an ``averaged'' second Chern number.

\begin{remark} \label{remark: local mean dimension}
\cite[Theorem 1.2]{Matsuo--Tsukamoto 2} asserts 
\[ \dim_{loc}(\moduli_d:\mathbb{R}) = 8\rho(d), \quad (d\in [0,+\infty)\setminus \mathcal{D}).\]
Here $\dim_{loc}(\moduli_d:\mathbb{R})$ is the \textbf{local mean dimension} of $\moduli_d$.
Local mean dimension is a variant of mean dimension, and it is always a lower bound on the original 
mean dimension. Therefore we get (\ref{eq: lower bound on the mean dimension}).
\end{remark}

Corollary \ref{cor: main corollary} is  the second success of 
non-trivial calculation of mean dimension in geometric analysis.
The first one was found by Matsuo--Tsukamoto \cite[Corollary 1.2]{Matsuo--Tsukamoto Brody curve}.
They proved the formula of the mean dimension of the system of Lipschitz holomorphic curves in the Riemann sphere.
In the case of holomorphic curves the Nevanlinna theory provides a very simple method for obtaining the upper bound on mean dimension
(\cite{Tsukamoto-nagoya}).
So the difficult part of \cite[Corollay 1.2]{Matsuo--Tsukamoto Brody curve} is the proof of the lower bound.
But, in the Yang--Mills case, the upper bound (Theorem \ref{thm: main theorem}) is 
also difficult because we don't have a ``Nevanlinna theory'' for ASD equation.
We need to develop a entirely new technique to obtain the upper bound, and this is the main task of the paper.
The outline of the proof is explained in Section \ref{subsection: ideas of the proof}.
Here we emphasize a key idea of the proof; using \textbf{metric mean dimension}.
Metric mean dimension is a notion introduced by Lindenstrauss--Weiss \cite{Lindenstrauss--Weiss}.
It is a bridge between topological entropy theory and mean dimension theory.
We review its definition in Section \ref{subsection: review of mean dimension}.
In this paper we show that metric mean dimension is a very flexible tool for obtaining a good upper bound 
on mean dimension.
Probably no one has expected that metric mean dimension is useful in geometric analysis.
So this is the most important point of the paper.
Hopefully this idea has a potential to 
be applied to many other problems.
For example, Gromov \cite[Chapter 4]{Gromov} studied a dynamical system consisting of 
complex subvarieties in $\mathbb{C}^n$.
He proved an upper bound on the mean dimension \cite[p. 408, Corollary]{Gromov}.
But his estimate is very crude. 
So he proposed the problem of proving a better bound
\cite[p. 409, Remarks and open questions (a)]{Gromov}.
It seems difficult to reach a good estimate by improving Gromov's argument directly.
Metric mean dimension might shed a new light on this problem.

\subsection{Application to dynamical embedding problem} \label{subsection: application to dynamical embedding problem}

Here we discuss one application of Theorem \ref{thm: main theorem}
in order to illustrate a dynamical importance of mean dimension.
In this subsection we restrict the $\mathbb{R}$-action 
(\ref{eq: action on the moduli}) to the subgroup $\mathbb{Z}\subset \mathbb{R}$, 
and we consider $\moduli_d$ as a space endowed with a continuous action of $\mathbb{Z}$.
The mean dimension $\dim(\moduli_d:\mathbb{Z})$ of this $\mathbb{Z}$-action is equal to 
$\dim(\moduli_d:\mathbb{R})$. So we get (Theorem \ref{thm: main theorem})
\[ \dim(\moduli_d:\mathbb{Z})\leq 8\rho(d).\]

Let $D$ be a natural number, and let $([0,1]^D)^{\mathbb{Z}}$ be the $\mathbb{Z}$-shift on the $D$-dimensional cube
(i.e. the ``$D$-dimensional version'' of the Hilbert cube).
$\mathbb{Z}$ naturally acts on this space, and
its mean dimension is $D$.
The following \textit{embedding problem} is a long-standing question in topological dynamics.
\begin{problem}
Let $M$ be a $\mathbb{Z}$-system, i.e. a compact metric space endowed with a continuous action of $\mathbb{Z}$.
Decide whether there exists a $\mathbb{Z}$-equivariant topological embedding from $M$ into the shift 
$([0,1]^D)^{\mathbb{Z}}$.
\end{problem}

This problem goes back to the Ph.D. thesis of Jaworski \cite{Jaworski} in 1974.
But here we skip the history and present only a current development.
If we can equivariantly embed $M$ into $([0,1]^D)^{\mathbb{Z}}$ then the mean dimension $\dim(M:\mathbb{Z})$ is 
less than or equal to $D$.
Lindenstrauss--Tsukamoto \cite{Lindenstrauss--Tsukamoto} conjectured that the following partial converse 
holds.
\begin{conjecture} \label{conjecture: embedding}
Let $M_n$ $(n\geq 1)$ be the space of periodic points of period $n$ in $M$.
Suppose 
\[ \dim(M:\mathbb{Z}) < \frac{D}{2}, \quad \frac{\dim M_n}{n} < \frac{D}{2} \quad (\forall n\geq 1).\]
Then we can embed $M$ into $([0,1]^D)^{\mathbb{Z}}$ equivariantly.
\end{conjecture}

Roughly speaking, we conjectured that mean dimension and periodic points are the only essential obstructions to the embedding.
This conjecture itself is widely open, but Gutman--Tsukamoto \cite{Gutman--Tsukamoto}
found that we can solve the problem if we sightly extend the system $M$ by using an aperiodic symbolic subshift.
Let $\{1,2,\dots,l\}^{\mathbb{Z}}$ be the symbolic shift, and let $Z\subset \{1,2,\dots,l\}^{\mathbb{Z}}$
be a subsystem without periodic points. 
We consider the product system $M\times Z$, which naturally admits a $\mathbb{Z}$-action and becomes an extension of the original 
system $M$.
The mean dimension of $M\times Z$ is equal to the mean dimension of $M$.
From \cite[Corollary 1.8]{Gutman--Tsukamoto}, we get:
\begin{theorem}
If the mean dimension $\dim(M:\mathbb{Z})$ is strictly smaller than $D/2$, then
we can embed the product system $M\times Z$ into $([0,1]^D)^{\mathbb{Z}}$ equivariantly.
\end{theorem}
Here the condition $\dim(M:\mathbb{Z})<D/2$ is known to be optimal (\cite[Proposition 4.2]{Gutman--Tsukamoto}).
By applying this theorem to $\moduli_d$, we get the following corollary.
\begin{corollary} \label{cor: embedding of moduli}
Suppose $\rho(d)< D/16$. Then $\moduli_d\times Z$ can be $\mathbb{Z}$-equivariantly embedded into $([0,1]^D)^{\mathbb{Z}}$.
\end{corollary}
This is a manifestation that the energy density $\rho(d)$ properly controls the size of $\moduli_d$.
If Conjecture \ref{conjecture: embedding} is proved, then we will be able to show that $\moduli_d$ itself can be embedded into 
$([0,1]^D)^{\mathbb{Z}}$ under the same condition $\rho(d)<D/16$.
Here it is worth to point out that we have no idea how to construct concretely the embedding given in Corollary 
\ref{cor: embedding of moduli}.
The above is a pure existence theorem.
It is very interesting to find an explicit construction of such an embedding because
it will give a new way to obtain an upper bound on the mean dimension;
if $\moduli_d\times Z$ can be equivariantly embedded into $([0,1]^D)^{\mathbb{Z}}$, then 
we get $\dim(\moduli_d:\mathbb{Z}) \leq D$.

\subsection{Ideas of the proof} \label{subsection: ideas of the proof}

In this subsection we explain a rough strategy of the proof of Theorem \ref{thm: main theorem}.
Our argument here is intuitive and non-rigorous. 

The most important idea is the use of metric mean dimension as we explained in 
the end of Section \ref{subsection: main result}.
Metric mean dimension is always an upper bound on mean dimension (Theorem \ref{thm: metric mean dimension}).
So we want to estimate the metric mean dimension of $\moduli_d$.
Intuitively this means that we estimate how many memory (/bit) we need in order to 
store on computer the orbits of $\moduli_d$ over the time $-T<t<T$ up to an error $\varepsilon>0$.
We want to know its asymptotics as $T\to \infty$ and $\varepsilon\to 0$.
Our argument has the following three steps.

\textbf{Step 1: Decomposition of $\moduli_d$}. 
We decompose the space $\moduli_d$ into appropriately small pieces:
\[ \moduli_d = U_1 \cup\dots \cup U_n.\]
We try to memorize each $U_i$ separately.
This is an advantage of metric mean dimension over original mean dimension.
Mean dimension does not behave smoothly for a decomposition of a space.
Metric mean dimension is flexible for such a decomposition if we appropriately control the number $n$ of the pieces.
So we can \textit{localize} the argument by using metric mean dimension.

\textbf{Step 2: Instanton approximation}.
The above $U_i$ are infinite dimensional in general.
We construct their finite dimensional approximations by
using the technique of \textbf{instanton approximation}. 
Instanton approximation is an analogue of the famous Runge theorem in complex analysis;
for any meromorphic function in $\mathbb{C}$ and any compact subset $K\subset \mathbb{C}$ 
we can construct a rational function which approximates the given 
function over $K$.
In the same spirit, for any ASD connection $A$ on $E$ and any compact subset $K\subset X$, 
we can construct an instanton (finite energy ASD connection)
which approximates $A$ over $K$.
Instanton approximation technique was first introduced by Taubes \cite{Taubes path} and Donaldson \cite{Donaldson approximation}, 
and it was used by Matsuo--Tsukamoto \cite{Matsuo--Tsukamoto} in the context of mean dimension.
Here we apply this technique to our present situation.
For each $U_i$ we construct a map 
\[ U_i \to V_i, \quad [A]\mapsto [A'], \]
such that $A'$ is an instanton which approximates $A$ over $-T<t<T$.
We can control the energy of $A'$ so that $V_i$ becomes a finite dimensional space.
$V_i$ is a good approximation of $U_i$ over $-T<t<T$.
So we only need to memorize $V_i$ instead of $U_i$.

\textbf{Step 3: Quantitative deformation theory}.
We investigate $V_i$ by constructing a deformation theory of instantons.
Instanton deformation theory is a quite standard subject, but our main emphasis is on its quantitative aspect.
We need to develop a deformation theory with estimates independent of 
several parameters (e.g. second Chern number, etc.).
A key ingredient is a decomposition of $\mathbb{R}$ into ``good intervals'' and ``bad intervals''.
(Indeed this decomposition will be also important in Step 1.)
We fix a sufficiently small number $\nu>0$.
Take an ASD connection $A$ on $E$, and let $n\in \mathbb{Z}$.
If the $L^\infty$-norm of the curvature $F_A$ over $n<t<n+1$ 
is greater than or equal to $\nu$, then we call the interval $(n,n+1)$ good.
Otherwise we call it bad.
If $A$ is an instanton, then there are only finitely many good intervals.
The meaning of this good/bad dichotomy is as follows.
If $(n,n+1)$ is good, then for any gauge transformation $g$ of $E$ over $n<t<n+1$ we have 
\[ \min_{\pm} \norm{g \pm 1}_{L^\infty((n,n+1)\times S^3)} 
    \leq \const(\nu) \cdot \norm{d_A g}_{L^2_{2,A}((n,n+1)\times S^3)}.\]
(See Lemma \ref{lemma: non-degeneracy estimate}.)
This means that we have a good control of gauge transformations over good intervals.
If $(n,n+1)$ is bad, then $A$ is close to a trivial flat connection (which is reducible) over $n<t<n+1$.
So we lose the above control of gauge transformations there.
This apparently causes a difficulty.
But if $A$ is close to a trivial flat connection, then its structure is simple.
So $A$ has little information over bad intervals.
(This means that bad intervals are ``not so bad''.)
We need to analyze these two different behaviors separately.
This can be done by introducing appropriate weighted norms, and 
we alway have to care effects of the weight on our estimates.

Our quantitative deformation theory tells us how many memory we need in order to memorize $V_i$.
Then we combine this with the results in the previous steps, 
and we can get the desired estimate on the metric mean dimension.

\textbf{Organization of the paper}:
In Section \ref{subsection: review of mean dimension} we explain the basic definitions of mean dimension 
and metric mean dimension.
In Section \ref{subsection: energy density} we prepare a lemma on the energy density $\rho(d)$.
In Section \ref{subsection: notations} we explain some notations which are used in the rest of the paper.

In Section \ref{subsection: setting} we introduce weighted norms which reflect the good/bad decomposition 
structure.
In Section \ref{subsection: main propositions} we state three main propositions 
(Propositions \ref{prop: decomposition}, \ref{prop: instanton approximation} and \ref{prop: quantitative deformation theory})
and prove Theorem \ref{thm: main theorem} by assuming them.
Propositions \ref{prop: decomposition}, \ref{prop: instanton approximation} and \ref{prop: quantitative deformation theory}
correspond to the above three steps respectively, and their proofs occupy the rest of the paper.

In Section \ref{section: decomposition of moduli_d} we prove Proposition \ref{prop: decomposition}.
In Section \ref{section: instanton approximation} we prepare several estimates on instanton approximation and prove 
Proposition \ref{prop: instanton approximation}.
In Section \ref{section: quantitative deformation theory} we develop a quantitative study of instanton deformation theory 
in detail and prove Proposition \ref{prop: quantitative deformation theory}.

\textbf{Acknowledgement}.
I wish to thank Professor Kenji Fukaya and Professor Elon Lindenstrauss.
I came up with the idea of using metric mean dimension through conversations with them.

\section{Some preliminaries}

\subsection{Review of mean dimension} \label{subsection: review of mean dimension}

In this subsection we review the basic facts on the mean dimension theory.
For the details, see Gromov \cite{Gromov} and Lindenstrauss--Weiss \cite{Lindenstrauss--Weiss}.

Let $(M,\dist)$ be a compact metric space. Here $\dist$ is a distance function of $M$.
We introduce some metric invariants of $(M,\dist)$.
Let $N$ be a topological space.
For $\varepsilon>0$, a continuous map $f:M\to N$ is called an $\varepsilon$-\textbf{embedding} 
if $\mathrm{Diam} f^{-1}(y) < \varepsilon$ for all $y\in N$.
We define the $\varepsilon$-\textbf{width dimension} $\widim_\varepsilon(M,\dist)$ 
as the minimum integer $n\geq 0$ such that 
there exist an $n$-dimensional finite polyhedron $P$ and an $\varepsilon$-embedding $f:M\to P$.
The covering dimension $\dim M$ is obtained by 
\[ \dim M  = \lim_{\varepsilon\to 0}\widim_\varepsilon(M,\dist).\]
For $\varepsilon>0$ we set 
\begin{equation*}
   \begin{split}
   &\#(M,\dist,\varepsilon) = \min\{\, |\alpha|\, |\, \text{$\alpha$ is an open covering of $M$ with 
   $\Diam U<\varepsilon$ for all $U\in \alpha$}\}, \\
   &\#_{\mathrm{sep}}(M,\dist,\varepsilon) = \max\{n\geq 1|\, \exists x_1,\dots,x_n\in M \text{ with }
   \dist(x_i,x_j)>\varepsilon \> (i\neq j)\}.
   \end{split}
\end{equation*} 
These are almost equivalent to each other:
For $0<\delta<\varepsilon/2$ 
\[ \#_{\mathrm{sep}}(M,\dist,\varepsilon) \leq \#(M,\dist,\varepsilon) 
   \leq \#_{\mathrm{sep}}(M,\dist,\delta).\]
The next lemma will be useful.
\begin{lemma} \label{lemma: separated set}
Let $(M,\dist)$ and $(N,\dist')$ be metric spaces.
Let $\varepsilon>0$ and $\delta>0$.
Suppose there exists a map (not necessarily continuous) $f:M\to N$ satisfying 
\[ \dist'(f(x),f(y)) \leq \delta \Rightarrow \dist(x,y) \leq \varepsilon.\]
Then $\#_{\mathrm{sep}}(M,\dist,\varepsilon)\leq \#_{\mathrm{sep}}(N,\dist',\delta)$.
\end{lemma}
\begin{proof}
Obvious. 
\end{proof}
The following example is important.
This was used by Li--Liang \cite[Lemma 7.4]{Li--Liang}.
\begin{example} \label{example: separated set of Banach ball}
Let $(V,\norm{\cdot})$ be an $n$-dimensional Banach space over $\mathbb{R}$.
Let $B_r(V)$ be the closed $r$-ball of $V$ around the origin. For any $\varepsilon>0$
\[ \#_{\mathrm{sep}}(B_r(V),\norm{\cdot},\varepsilon) \leq \left(\frac{\varepsilon+2r}{\varepsilon}\right)^n.\]
\end{example}
\begin{proof}
Let $\mu$ be the translation invariant measure (i.e. Haar measure) on $V$ normalized so that 
$\mu(B_1(V))=1$.
Then for any $r>0$ we have 
$\mu(B_r(V)) = r^n$.
Choose $\{x_1,\dots,x_N\}\subset B_r(V)$ with $\norm{x_i-x_j}>\varepsilon$ for $i\neq j$.
Let $B_i$ be the closed $\varepsilon/2$-ball centered at $x_i$.
These $B_i$ are disjoint and their union is contained in $B_{r+\varepsilon/2}(V)$.
Hence 
\[  N(\varepsilon/2)^n = \mu\left(\bigcup_{i=1}^N B_i\right) \leq \mu(B_{r+\varepsilon/2}(V)) = (r+\varepsilon/2)^n.\]
\end{proof}

Suppose the Lie group $\mathbb{R}$ continuously acts on a compact metric space $(M,\dist)$.
For a subset $\Omega\subset \mathbb{R}$ we define a new distance $\dist_\Omega$ on $M$ by 
\[ \dist_\Omega(x,y) := \sup_{t\in \Omega} \dist(t.x,t.y).\]
We define the mean dimension $\dim(M:\mathbb{R})$ by
\[ \dim(M:\mathbb{R}) := \lim_{\varepsilon\to 0}
    \left(\lim_{T\to +\infty} \frac{\widim_\varepsilon(M,\dist_{(-T,T)})}{2T}\right).\]
This is independent of the choice of a distance function $\dist$.
So the mean dimension is a topological invariant.
If $\dim M <+\infty$, then the mean dimension $\dim(M:\mathbb{R})$ is zero.

Next we introduce metric mean dimension
(Lindenstrauss--Weiss \cite[Section 4]{Lindenstrauss--Weiss}). 
For $\varepsilon>0$ we define $S(M,\dist,\varepsilon)$ by 
\[ S(M,\dist,\varepsilon) = \lim_{T\to +\infty} \frac{\log \#(M,\dist_{(-T,T)},\varepsilon)}{2T}.\]
This is the entropy of $M$ ``at the scale $\varepsilon$''.
The above limit always exists because of the natural subadditivity:
\[ \#(M,\dist_{\Omega_1\cup\Omega_2},\varepsilon) \leq \#(M,\dist_{\Omega_1},\varepsilon)
   + \#(M,\dist_{\Omega_2},\varepsilon), \quad (\Omega_1,\Omega_2\subset \mathbb{R}).\]
The topological entropy of $M$ is defined by 
$h_{\mathrm{top}}(M:\mathbb{R}) = \lim_{\varepsilon\to 0}S(M,\dist,\varepsilon)$.
We define the metric mean dimension $\mmdim(M,\dist:\mathbb{R})$ by 
\begin{equation}\label{eq: definition of meric mean dimension}
   \mmdim(M,\dist:\mathbb{R}) := \liminf_{\varepsilon\to 0}\frac{S(M,\dist,\varepsilon)}{|\log\varepsilon|}.
\end{equation}
The metric mean dimension $\mmdim(M,\dist:\mathbb{R})$ depends on the choice of a distance.
If the topological entropy is finite, then the metric mean dimension is zero.
Lindenstrauss--Weiss \cite[Theorem 4.2]{Lindenstrauss--Weiss} proved the following fundamental theorem.
\begin{theorem} \label{thm: metric mean dimension}
Metric mean dimension is always an upper bound on mean dimension:
\[ \dim(M:\mathbb{R}) \leq \mmdim(M,\dist:\mathbb{R}).\]
In particular if the topological entropy is finite, then the mean dimension is zero.
\end{theorem}

\subsection{Energy density}  \label{subsection: energy density}

In this subsection we prepare a lemma on the energy density $\rho(d)$ introduced in 
(\ref{eq: energy density}).
First of all, the limit in the definition (\ref{eq: energy density}) always exists because 
we have the natural subadditivity:
\begin{equation*}
 \sup_{t\in \mathbb{R}}\int_{(t,t+T_1+T_2)\times S^3}|F_A|^2d\vol \leq 
  \sup_{t\in \mathbb{R}}\int_{(t,t+T_1)\times S^3}|F_A|^2d\vol +
  \sup_{t\in \mathbb{R}}\int_{(t,t+T_2)\times S^3}|F_A|^2d\vol.
\end{equation*}

\begin{lemma} \label{lemma: another form of energy density}
\begin{equation}  \label{eq: energy density another form}
   \rho(d) = 
   \lim_{T\to +\infty}\left(\frac{1}{16\pi^2T}\sup_{[A]\in \moduli_d}\int_{(-T,T)\times S^3}|F_A|^2d\vol\right).
\end{equation}
The limit of the right-hand-side exists because of the subadditivity.
\end{lemma}
\begin{proof}
This can be proved by the method of \cite[Theorem 1.3]{Tsukamoto-energy density}.
But here we give a simpler proof based on the ergodic theorem.
In this proof we restrict the $\mathbb{R}$-action (\ref{eq: action on the moduli}) to 
the subgroup $\mathbb{Z}\subset \mathbb{R}$  
as in Section \ref{subsection: application to dynamical embedding problem}.
We denote by $\rho_1(d)$ the right-hand-side of (\ref{eq: energy density another form}).
$\rho(d)\leq \rho_1(d)$ is obvious.
We define a continuous function $\varphi: \moduli_d\to \mathbb{R}$ by 
\[ \varphi([A]) = \frac{1}{8\pi^2}\int_{(0,1)\times S^3}|F(A)|^2d\vol.\]
Then for $[A]\in \moduli_d$ and positive integers $n$ we have the following equation:
\[  \frac{1}{8\pi^2 n}\int_{(0,n)\times S^3}|F(A)|^2d\vol = \frac{1}{n}\sum_{k=0}^{n-1}\varphi(k[A]).\]
Here $k[A]=[k^*A]$ is the pull-back of $[A]$ by $(t,\theta)\mapsto (t+k,\theta)$.
We can choose a sequence $[A_1],[A_2],\dots$ in $\moduli_d$ so that 
\[ \frac{1}{n}\sum_{k=0}^{n-1}\varphi(k[A_n])
  = \frac{1}{8\pi^2 n}\int_{(0,n)\times S^3}|F(A_n)|^2d\vol \to \rho_1(d) \quad (n\to \infty).\]
We define a Borel probability measure $\mu_n$ on $\moduli_d$ by 
\[ \mu_n := \frac{1}{n}\sum_{k=0}^{n-1} \delta_{k[A_n]} \]
where $\delta_{k[A_n]}$ is the delta measure concentrated at the point $k[A_n]$.
Then 
\[ \int_{\moduli_d}\varphi\, d\mu_n = \frac{1}{n}\sum_{k=0}^{n-1}\varphi(k[A_n]) \to \rho_1(d).\]
The space of Borel probability measures is weak$^*$-compact.
So we can pick up an accumulation point 
$\mu_\infty$ of $\{\mu_n\}$. 
$\mu_\infty$ is a $\mathbb{Z}$-invariant Borel probability measure 
(Einsiedler--Ward \cite[Theorem 4.1]{Einsiedler--Ward}) and satisfies 
\[ \int_{\moduli_d}\varphi \, d\mu_\infty = \rho_1(d).\]
By the ergodic decomposition \cite[Theorem 4.8]{Einsiedler--Ward}, 
we can choose an ergodic component $\mu$ of $\mu_\infty$ satisfying 
\[ \int_{\moduli_d}\varphi\, d\mu \geq \rho_1(d).\]
By the pointwise ergodic theorem \cite[Theorem 2.30]{Einsiedler--Ward}, for $\mu$-a.e. $[A]\in \moduli_d$
\[ \frac{1}{n}\sum_{k=0}^{n-1} \varphi(k[A]) \to \int_{\moduli_d}\varphi \, d\mu \geq \rho_1(d).\]
This implies $\rho(A)\geq \rho_1(d)$ for $\mu$-a.e. $[A]\in \moduli_d$.
In particular we get $\rho(d)\geq \rho_1(d)$.
\end{proof}

\subsection{Notations} \label{subsection: notations}

$\bullet$
In most of the arguments the variable $t$ means the natural projection $t:\mathbb{R}\times S^3\to \mathbb{R}$.

$\bullet$
The value of $d$ (which is the parameter of $\moduli_d$) is fixed in the 
rest of the paper.
So we treat it as a constant and omit to write the dependence on $d$ in various estimates.
For two quantities $x$ and $y$ we write 
\[ x \lesssim y \]
if there exists a universal positive constant $C$ (which might depend on $d$) satisfying 
$x\leq C y$.
We also use the following notation:
\[ x \lesssim_{a,b,c,\dots, k} y \]
This means that there exists a positive constant $C(a,b,c,\dots,k)$ 
which depends only on parameters $a,b,c,\dots,k$ satisfying $x\leq C(a,b,c,\dots,k)y$.

$\bullet$
Let $A$ be a connection on $E$. Let $k\geq 0$ be an integer, and let $p\geq 1$.
For $\xi\in \Omega^i(\ad E)$ $(0\leq i\leq 4)$
and a subset $U\subset X$, we define a norm $\norm{\xi}_{L^p_{k,A}(U)}$ by 
\[ \norm{\xi}_{L^p_{k,A}(U)} := \left(\sum_{j=0}^k \norm{\nabla_A^j\xi}_{L^p(U)}^p\right)^{1/p}.\]
For $\alpha<\beta$ we often denote the norm $\norm{\xi}_{L^p_{k,A}((\alpha,\beta)\times S^3)}$
by $\norm{\xi}_{L^p_{k,A}(\alpha,\beta)}$.

\section{Main propositions and the proof of Theorem \ref{thm: main theorem}}  \label{section: main propositions}

\subsection{Setting of the weighted norms} \label{subsection: setting}

The following lemma 
is a basis of our good/bad decomposition argument.
\begin{lemma} \label{lemma: exponential decay}
We can choose $\nu>0$ so that the following statement holds.
Let $T>1$ (possibly $T=\infty$) and let $A$ be an ASD connection on $E$ over $(0,T)\times S^3$ satisfying 
$\norm{F_A}_{L^\infty(0,T)} < \nu$.
Then 

\noindent 
(1) $|F_A|\lesssim \exp(2|t-T/2|-T)$ over $1/3<t<T-1/3$.
Moreover $\norm{F_A}_{L^2(0,T)}<1$.

\noindent 
(2) There exists a bundle trivialization $g$ of $E$ over $0<t<T$ such that
\begin{itemize}
  \item $g$ is a temporal gauge, i.e. the connection matrix $g(A)$ has no $dt$-component.
  \item $|\nabla^k g(A)| \lesssim_k \exp(2|t-T/2|-T)$ over $1/3<t<T-1/3$ for all integers $k\geq 0$.
\end{itemize}
\end{lemma}
\begin{proof}
This can be proved in the same way as in Donaldson--Kronheimer \cite[Chapter 7.3, Proposition 7.3.3]{Donaldson--Kronheimer} 
or Donaldson \cite[Proposition 4.4]{Donaldson}.
But here we briefly explain how to deduce the above statement from these references.

(1)
By \cite[Proposition 4.4]{Donaldson} we can find $L>0$ and $\nu>0$ such that if an ASD connection $A$ over $-L<t<L$ satisfies 
$\norm{F_A}_{L^\infty(-L,L)}<\nu$ then 
\[ \int_{-1<t<1}|F_A|^2 d\vol < \frac{1}{10}\left(\int_{-L<t<-L+1}|F_A|^2 d\vol +\int_{L-1<t<L}|F_A|^2 d\vol\right).\]
Using this estimate iteratively, we can show that the condition $\norm{F_A}_{L^\infty(0,T)}<\nu \ll 1$ implies 
$\norm{F_A}_{L^2(0,T)}\lesssim \nu$ (the implicit constant is independent of $T$).
Then we can prove the exponential decay of the condition (1) by \cite[Proposition 7.3.3]{Donaldson--Kronheimer}.

(2) The derivatives of $F_A$ also satisfy the same exponential decay condition.
Then we can choose a bundle trivialization $g$ of $E$ over $\{t=T/2\}$ so that 
$|\nabla^k g(A)| \lesssim_k e^{-T}$. We extend it to $-T<t<T$ by the temporal gauge condition.
This satisfies the required properties.
\end{proof}

For a real number $t$ and a subset $G$ of $\mathbb{Z}$ we define
$|t-G|$ as the infimum of $|t-n|$ over $n\in G$.
Let $A$ be a connection on $E$.
We set 
\[ G(A) = \{n\in \mathbb{Z}| \norm{F_A}_{L^\infty(n,n+1)}\geq \nu\}.\] 
Here $\nu$ is the positive constant introduced in Lemma \ref{lemma: exponential decay}.
For a positive integer $T$ we set $G(A,T) = G(A)\cup\{-T,T\}$.
For $r>0$ we define $U_r(A,T)\subset \moduli_d$ as the set of 
$[B]\in \moduli_d$ such that there exists a gauge transformation $g$ of $E$ over $-T<t<T$ satisfying 
\[ e^{|n-G(A,T)|} \norm{g(B)-A}_{L^2_{10,A}(n,n+1)} \leq r \text{ for all integers $-T\leq n\leq T-1$}.\]

Let $A$ be a non-flat instanton (finite energy ASD connection) on $E$.
Here ``finite energy'' means 
\[ \int_{X}|F_A|^2 d\vol < +\infty.\]
By \cite[Theorem 4.2]{Donaldson} the curvature $F_A$ decays exponentially as $t\to \pm \infty$.
We define $G'(A)$ as the set of integers $n$ satisfying $\norm{F_A}_{L^\infty(n,n+1)}\geq \nu/2$.
This is a non-empty finite set. Fix $0<\alpha<1$ and 
we define a smooth function $W_A:\mathbb{R}\to (0,+\infty)$
as a smoothing of the function 
\[ \exp(\alpha |t-G'(A)|).\]
The function $\exp(\alpha|t-G'(A)|)$ has finitely many non-differentiable points. So we smooth them out.
Details of the smoothing are not important.
We construct $W_A$ so that it satisfies 
\[ e^{\alpha|t-G'(A)|} \lesssim W_A(t) \lesssim e^{\alpha |t-G'(A)|}, \quad 
   W_A^{(k)} \lesssim_k W_A,\]
where the implicit constants are independent of $t\in \mathbb{R}$. 
$W_A^{(k)}$ is the $k$-th derivative of $W_A$.

Let $G'(A) = \{n_1<n_2<\dots<n_G\}$, and set $n_0=-\infty$ and $n_{G+1}=+\infty$.
For $u\in \Omega^i(\ad E)$ and $k\geq 0$ we define a norm 
\begin{equation} \label{eq: defintion of weighted norm}
   \nnorm{u}_{k,A} = \max_{0\leq j\leq G}\norm{W_A u}_{L^2_{k,A}(n_j,n_{j+1})}.
\end{equation}
For $r>0$ we define $V_r(A)$ as the set of gauge equivalence classes of ASD connections $B$ on $E$
such that there exists a gauge transformation $g$ of $E$ satisfying 
\[ \nnorm{g(B)-A}_{2,A} \leq r.\]

\subsection{Main propositions and the proof of Theorem \ref{thm: main theorem}} 
\label{subsection: main propositions}

\begin{proposition} \label{prop: decomposition}
For any $\delta>0$ and any integer $T>1$ there exist 
$[A_1],\dots,[A_n]\in \moduli_d$ satisfying 
\[ \log n\lesssim_\delta T, \quad \moduli_d = \bigcup_{i=1}^n U_\delta(A_i,T).\]
\end{proposition}

\begin{proposition} \label{prop: instanton approximation}
For any $r>0$ we can choose $\delta_0=\delta_0(r)>0$ satisfying the following statement.
For any $[A]\in \moduli_d$ and any integer $T>1$ there exists a non-flat instanton $A'$ on $E$ and a map 
\[ U_{\delta_0}(A,T)\to V_r(A'), \quad [B]\mapsto [B'] \]
such that 

\noindent 
(1) 
\[ \norm{F_{A'}}_{L^\infty(X)} \leq D_0, \quad 
  \left| \int_X |F_{A'}|^2d\vol -\int_{(-T,T)\times S^3} |F_A|^2d\vol \right| \lesssim 1.\]
Here $D_0$ is a universal constant independent of $r$.

\noindent
(2) For any $[B]\in U_{\delta_0}(A,T)$ there exists a gauge transformation $g$ of $E$ over $|t|<T-1$ satisfying
\[ |g(B')-B|\lesssim e^{-\sqrt{2}|t-T|} + e^{-\sqrt{2}|t+T|} \quad (|t|<T-1).\]
\end{proposition}

For two connections $A_1$ and $A_2$ on $E$, we set
\[ \dist_{L^\infty}([A_1],[A_2]) = \inf_{g:E\to E} \norm{g(A_1)-A_2}_{L^\infty(X)} ,\]
where $g$ runs over all gauge transformations of $E$.

\begin{proposition} \label{prop: quantitative deformation theory}
For any $D>0$ there exist positive numbers $r_0=r_0(D)$ and $C_0=C_0(D)$ satisfying the following statement.
Let $A$ be a non-flat instanton on $E$ with $\norm{F_A}_{L^\infty(X)}\leq D$.
Then for any $0<\varepsilon<1$ 
\[ \#_{\mathrm{sep}}(V_{r_0}(A),\dist_{L^\infty},\varepsilon) \leq (C_0/\varepsilon)^{8c_2(A)+3},\]
where 
\[ c_2(A) = \frac{1}{8\pi^2}\int_X |F_A|^2d\vol.\]
\end{proposition}

The proofs of the above three propositions occupy the rest of the paper.
Here we prove Theorem \ref{thm: main theorem}, assuming them.

\begin{proof}[Proof of Theorem \ref{thm: main theorem}]
We define a distance on $\moduli_d$ by 
\[ \dist([A],[B]) = \inf_{g:E\to E}\norm{g(A)-B}_{L^\infty(0,1)},\]
where $g$ runs over all gauge transformations of $E$.
This is compatible with the given topology of $\moduli_d$.
Recall that for a subset $\Omega\subset \mathbb{R}$ we denote 
by $\dist_{\Omega}([A],[B])$ the supremum of 
$\dist([s^*A],[s^*B])$ over $s\in \Omega$.
We will prove the upper bound on the metric mean dimension: $\mmdim(\moduli_d,\dist:\mathbb{R})\leq 8\rho(d)$.
Then we get $\dim(\moduli_d:\mathbb{R})\leq 8\rho(d)$ since the metric mean dimension 
is an upper bound on the mean dimension (Theorem \ref{thm: metric mean dimension}).

Let $D_0>0$ be the universal constant introduced in Proposition \ref{prop: instanton approximation} (1), 
and let $r_0 = r_0(D_0)$ be the positive constant introduced in 
Proposition \ref{prop: quantitative deformation theory} with respect to $D_0$.
Moreover let $\delta_0=\delta_0(r_0(D_0))$ be the positive constant introduced in 
Proposition \ref{prop: instanton approximation} with respect to $r_0(D_0)$.
\begin{claim} \label{claim: counting in U_delta}
There exists $C_1>0$ satisfying the following statement.
For any $0<\varepsilon<1$ there exists an integer $L_0 = L_0(\varepsilon)>1$ 
such that for any integer $T>L_0$ and any $[A]\in \moduli_d$
we have 
\[ \log \#_{\mathrm{sep}}(U_{\delta_0}(A,T),\dist_{(-T+L_0,T-L_0)},\varepsilon) 
   \leq (|\log \varepsilon|+C_1)\left({\frac{1}{\pi^2}\int_{(-T,T)\times S^3}|F_A|^2d\vol + C_1}\right).\]
\end{claim}
\begin{proof}
By Proposition \ref{prop: instanton approximation} for any $[A]\in \moduli_d$ and any integer $T>1$ there
exist a non-flat instanton $[A']$ and a map 
\[ U_{\delta_0}(A,T)\to V_{r_0}(A'), \quad [B]\mapsto [B'] \]
satisfying the conditions (1) and (2) of the statement there.
If we choose $L_0=L_0(\varepsilon)>0$ sufficiently large, then (by the condition (2)) 
for any $[B]\in U_{\delta_0}(A,T)$ there exists 
a gauge transformation $g$ of $E$ over $|t|<T-1$ satisfying 
\[ |g(B')-B| <\varepsilon/3 \quad (|t|<T-L_0+1).\]
Then for any $[B_1],[B_2]\in U_{\delta_0}(A,T)$ with $T>L_0$ we get 
\[  \dist_{L^\infty}([B_1'],[B_2']) \leq  \varepsilon/3  \Longrightarrow 
    \dist_{(-T+L_0,T-L_0)}([B_1],[B_2]) \leq \varepsilon.\]
By Lemma \ref{lemma: separated set} 
\begin{equation*}
   \begin{split}
   \#_{\mathrm{sep}}(U_{\delta_0}(A,T),\dist_{(-T+L_0,T-L_0)},\varepsilon) &\leq 
   \#_{\mathrm{sep}}(V_{r_0}(A'),\dist_{L^\infty},\varepsilon/3) \\
   &\leq (3C_0/\varepsilon)^{8c_2(A')+3} \quad (\text{by Proposition \ref{prop: quantitative deformation theory}}).
   \end{split}
\end{equation*}
By the condition (1) of Proposition \ref{prop: instanton approximation} 
\[ 8c_2(A') +3 \leq \frac{1}{\pi^2} \int_{(-T,T)\times S^3}|F_A|^2 d\vol + \const,\]
where $\const$ is a universal constant.
Thus we get the conclusion.
\end{proof}

Take $0<\varepsilon<1$ and let $L_0=L_0(\varepsilon)>0$ be the positive number introduced in the above 
claim.
By Proposition \ref{prop: decomposition} for any integer $T>1$ there exist 
$[A_1],\dots,[A_n]\in \moduli_d$ satisfying 
\[ \log n \lesssim T+L_0, \quad \moduli_d = \bigcup_{i=1}^n U_{\delta_0}(A_i,T+L_0).\]
Then  $\#(\moduli_d,\dist_{(-T,T)},\varepsilon)$ is bounded by 
\[  \sum_{i=1}^n \#(U_{\delta_0}(A_i,T+L_0),\dist_{(-T,T)},\varepsilon) \leq 
    \sum_{i=1}^n \#_{\mathrm{sep}}(U_{\delta_0}(A_i,T+L_0),\dist_{(-T,T)},\varepsilon/3).\]
By Claim \ref{claim: counting in U_delta},
$\log \#(\moduli_d,\dist_{(-T,T)},\varepsilon)$ is bounded by 
\[   \log n + (|\log\varepsilon|+\log 3 + C_1)\left(\frac{1}{\pi^2}\sup_{[A]\in \moduli_d}
   \int_{(-T-L_0,T+L_0)\times S^3} |F_A|^2d\vol + C_1\right). \]
Since $\log n\lesssim T+L_0$ and $L_0$ does not depend on $T$, we get 
(by using Lemma \ref{lemma: another form of energy density})
\[ S(\moduli_d,\dist,\varepsilon) 
  = \lim_{T\to \infty} \frac{\log \#(\moduli_d,\dist_{(-T,T)},\varepsilon)}{2T}
  \leq \const + (|\log\varepsilon|+\log 3 + C_1)8\rho(d).\]
Here $\const$ and $C_1$ are independent of $\varepsilon$. Thus 
\[ \mmdim(\moduli_d,\dist:\mathbb{R}) 
  = \liminf_{\varepsilon \to 0}\frac{S(\moduli_d,\dist,\varepsilon)}{|\log\varepsilon|} \leq 8\rho(d).\]
\end{proof}

\section{Decomposition of $\moduli_d$: proof of Proposition \ref{prop: decomposition}}
\label{section: decomposition of moduli_d}

We prove Proposition \ref{prop: decomposition} in this section.
A theme of this section is a problem of gluing gauge transformations.
A simplified situation is the following:
Let $[A],[B]\in \moduli_d$. Let $U_1,U_2\subset X$ be open sets, 
and let $g_i$ be gauge transformations of $E$ over $U_i$ ($i=1,2$).
Suppose $|g_i(B)-A|$ are very small over $U_i$ for both $i=1,2$. 
Can we find a gauge transformation $h$ of $E$ over $U_1\cup U_2$ satisfying $|h(B)-A|\ll 1$?
Unfortunately the answer is \textit{No} in general.
If $A$ and $B$ are very close to flat connections over $U_1\cap U_2$, then we have to consider a \textit{gluing parameter} 
over $U_1\cap U_2$ and cannot find such a gauge transformation $h$. 
(This phenomena appears in constructions of gluing instantons.
See \cite[Chapter 7.2]{Donaldson--Kronheimer}.)
In Lemmas \ref{lemma: gluing gauge transformations non-degenerate} 
and \ref{lemma: gluing gauge transformations degenerate} below
we formulate situations where the answer to the above question becomes \textit{Yes}.

The following is a basis of the argument.
This is proved in \cite[Corollary 6.3]{Matsuo--Tsukamoto}.
\begin{lemma} \label{lemma: non-flat implies irreducible}
Let $A$ be a non-flat ASD connection on $E$ with $\norm{F_A}_{L^\infty}<\infty$.
Then $A$ is irreducible, i.e. if a gauge transformation $g$ satisfies $g(A)=A$ then 
$g=\pm 1$.
\end{lemma}
\begin{proof}
We give a sketch of the proof for the convenience of readers. 
Suppose $A$ is reducible. 
Then $A$ is reduced to a $U(1)$ connection.
In particular $F_A$ is a $u(1)$-valued anti-self-dual 2-form.
Using the Yang--Mills equation $d_A^*F_A=0$ and the Weitzenb\"{o}ck formula 
(see (\ref{eq: Weitzenbock formula}) in Section \ref{section: instanton approximation}),
we get 
\[ (\nabla^*\nabla + 2)F_A=0.\]
Here we have used the fact that the curvature $F_A$ does not contribute to the formula because 
it is $u(1)$-valued.
Then $\norm{F_A}_{L^\infty}<\infty$ implies $F_A=0$ all over $X$. 
See discussions around (\ref{eq: Green kernel estimate}).
\end{proof}

The next lemma means that we have a good control of gauge transformations over ``good intervals''.
\begin{lemma} \label{lemma: non-degeneracy estimate}
Let $\kappa>0$ and let $[A]\in \moduli_d$ with $\norm{F_A}_{L^\infty(0,1)}\geq \kappa$.
For any gauge transformation $g$ of $E$ over $0<t<1$ we have 
\[ \min_{\pm} \norm{g\pm 1}_{L^\infty(0,1)}\lesssim_{\kappa} \norm{d_A g}_{L^2_{2,A}(0,1)}.\]
\end{lemma}
\begin{proof}
It is standard that we can deduce this kind of statement from the following linearized one.
(For the detail, see \cite[Lemma 3.2]{Matsuo--Tsukamoto 2}.)

\begin{claim} \label{claim: linear non-degeneracy estimate}
Let $u$ be a section of $\ad E$ over $0<t<1$. Then 
\[ \norm{u}_{L^\infty(0,1)}\lesssim_\kappa \norm{d_A u}_{L^2_{2,A}(0,1)}.\]
\end{claim}
\begin{proof}
Suppose the contrary. Then there exist $[A_n]\in \moduli_d$ with 
$\norm{F_{A_n}}_{L^\infty(0,1)}\geq \kappa$ and $u_n\in \Gamma((0,1)\times S^3,\ad E)$ 
$(n\geq 1)$ satisfying 
\[ \norm{d_{A_n}u_n}_{L^2_{2,A_n}(0,1)} < \frac{1}{n}, \quad \norm{u_n}_{L^\infty(0,1)} =1.\]
Since $\moduli_d$ is compact, we can assume that $A_n$ converges to some $A$
with $\norm{F_A}_{L^\infty(0,1)}\geq \kappa$ in $C^\infty$ over every compact subset.
Then $\{u_n\}$ is bounded in $L^2_{3,A}((0,1)\times S^3)$. By choosing a subsequence, we can assume that 
$u_n$ weakly converges to some $u$ in $L^2_{3,A}((0,1)\times S^3)$ with $d_A u=0$.
We have $\norm{u}_{L^\infty(0,1)}=1$ because
the Sobolev embedding $L^2_{3,A}((0,1)\times S^3)\to L^\infty((0,1)\times S^3)$ is compact.
This means that $A$ is reducible over $0<t<1$.
By the unique continuation theorem (Donaldson--Kronheimer \cite[Chapter 4, Lemma 4.3.21]{Donaldson--Kronheimer}) 
$A$ is reducible all over $X$.
This contradicts Lemma \ref{lemma: non-flat implies irreducible}.
\end{proof}
\end{proof}

In the next two lemmas we formulate situations where we can glue two gauge transformations.
In the first lemma, an overlapping region is ``good''. The argument is straightforward.
In the second lemma, an overlapping region is ``bad''. Our formulation have to be more involved.

\begin{lemma} \label{lemma: gluing gauge transformations non-degenerate}
For any $\kappa, \delta>0$ we can choose $\varepsilon_1 = \varepsilon_1(\kappa,\delta)>0$ so that the following 
statement holds.
Let $[A],[B]\in \moduli_d$, and let $g_1$ and $g_2$ be gauge transformations of $E$ over 
$0<t<2$ and $1<t<3$ respectively. Suppose
\[ \norm{F_A}_{L^\infty(1,2)}\geq \kappa, \quad 
   \norm{g_i(B)-A}_{L^2_{10,A}(1,2)} < \varepsilon_1 \> (i=1,2).\]
Then there exists a gauge transformation $h$ of $E$ over $0<t<3$ such that 
$h=g_1$ over $0<t<1$, $h=\pm g_2$ over $2<t<3$ and 
\[ \norm{h(B)-A}_{L^2_{10,A}(1,2)} < \delta.\]
\end{lemma}
\begin{proof}
Set $w=g_2 g_1^{-1}$ over $1<t<2$. We have 
$d_A w = w\cdot (g_1(B)-A)+(A-g_2(B))\cdot w$. Hence 
$\norm{d_A w}_{L^2_{10,A}(1,2)}\lesssim \varepsilon_1$.
By Lemma \ref{lemma: non-degeneracy estimate} we get 
$\min_\pm \norm{w\pm 1}_{L^\infty(1,2)}\lesssim_\kappa \varepsilon_1$.
We can assume $\norm{w-1}_{L^\infty(1,2)}\leq \norm{w+1}_{L^\infty(1,2)}$.
Then $\norm{w-1}_{L^\infty(1,2)}\lesssim_\kappa \varepsilon_1 \ll 1$.
Thus $w$ is expressed as $w=e^u$ with $\norm{u}_{L^2_{11,A}(1,2)}\lesssim_\kappa \varepsilon_1$.
Take a cut-off $\varphi:\mathbb{R}\to [0,1]$ such that $\supp(d\varphi)\subset (0,1)$, $\varphi(0)=0$ 
and $\varphi(1)=1$.
We set $h=e^{\varphi u} g_1$. If we choose $\varepsilon_1$ sufficiently small, then this satisfies the statement.
\end{proof}

In the rest of this section we take and fix a point $\theta_0\in S^3$.
Recall that we introduced the positive constant $\nu$ 
in Lemma \ref{lemma: exponential decay}.
\begin{lemma} \label{lemma: gluing gauge transformations degenerate}
For any $\delta>0$ we can choose positive numbers $\varepsilon_2=\varepsilon_2(\delta)$ and 
$L_1=L_1(\delta)$ so that the following statement holds.
Take $[A],[B]\in \moduli_d$, an integer $T\geq 2L_1$ and gauge transformations 
$g_1$ and $g_2$ over $-1<t<L_1$ and $T-L_1<t<T+1$ respectively.
Suppose the following three conditions.
\begin{itemize}
 \item $\norm{F_A}_{L^\infty(0,T)}, \norm{F_B}_{L^\infty(0,T)}<\nu$, 
 $\norm{F_A}_{L^\infty(-1,0)}, \norm{F_A}_{L^\infty(T,T+1)}\geq \nu$.
 \item $\norm{g_1(B)-A}_{L^2_{10,A}(0,L_1)}, \norm{g_2(B)-A}_{L^2_{10,A}(T-L_1,T)}<\varepsilon_2$.
 \item Set $p=(L_1-1,\theta_0), q=(T-L_1+1,\theta_0)\in X$ and define 
 $g'_2(q):E_p \to E_p$ by the following commutative diagram:
  \begin{equation*}
     \begin{CD}
     E_p @>{\text{parallel translation by $B$}}>> E_q \\
     @VV{g'_2(q)}V @VV{g_2(q)}V \\
     E_p @>{\text{parallel translation by $A$}}>> E_q
     \end{CD}
  \end{equation*}
 Here the horizontal arrows are the parallel translations by $B$ and $A$ along the minimum geodesic between 
 $p$ and $q$. 
 Under these settings, we have 
 \[ \min_{\pm} \dist_{SU(2)}(g_1(p),\pm g'_2(q)) < \varepsilon_2,\]
 where $\dist_{SU(2)}$ is the distance on $SU(2)$ defined by the standard Riemannian structure.
\end{itemize}
Then there exists a gauge transformation $h$ of $E$ over $-1<t<T+1$ such that 
$h=g_1$ over $-1<t<0$, $h=\pm g_2$ over $T<t<T+1$ and 
\[ e^{\min(n+1,T-n)}\norm{h(B)-A}_{L^2_{10,A}(n,n+1)} < \delta \]
for all integers $0\leq n\leq T-1$.
\end{lemma}
\begin{proof}
Let $g_A$ and $g_B$ be the temporal gauges of $A$ and $B$ over $0<t<T$ introduced in Lemma \ref{lemma: exponential decay}.
The connection matrices $A':=g_A(A)$ and $B':=g_B(B)$ satisfy 
$|\nabla^k A'|, |\nabla^k B'|\lesssim_k \exp(2|t-T/2|-T)$.
Set $w_1=g_A\circ g_1\circ g_B^{-1}$ over $0<t<L_1$ and 
$w_2=g_A\circ g_2\circ g_B^{-1}$ over $T-L_1<t<T$.
They satisfy 
$\norm{w_1(B')-A'}_{L^2_{10,A'}(0,L_1)}<\varepsilon_2$ and $\norm{w_2(B')-A'}_{L^2_{10,A'}(T-L_1,T)}<\varepsilon_2$.
Moreover we can assume $\dist_{SU(2)}(w_1(p),w_2(q))<\varepsilon_2$.
Here we regard $w_1$ and $w_2$ as $SU(2)$-valued functions over $0<t<L_1$ and $T-L_1<t<T$ respectively.

We get $|dw_1|\lesssim \varepsilon_2+e^{-2L_1}$ and $|dw_2|\lesssim \varepsilon_2+e^{-2L_1}$ over 
$L_1-2<t<L_1$ and $T-L_1<t<T-L_1+2$ respectively.
Then $w_1$ and $w_2$ are expressed as $w_1=w_1(p)e^{u_1}$ over $L_1-2<t<L_1$ and 
$w_2=w_2(q)e^{u_2}$ over $T-L_1<t<T-L_1+2$ such that 
\[ \norm{u_1}_{L^2_{11}(L_1-2,L_1)} \lesssim \varepsilon_2+e^{-2L_1}, \quad 
   \norm{u_2}_{L^2_{11}(T-L_1,T-L_1+2)}\lesssim \varepsilon_2+e^{-2L_1}.\]

We take a path $v:\mathbb{R}\to SU(2)$ such that $v(t)=w_1(p)$ for $t\leq L_1-1$,
$v(t)=w_2(q)$ for $t\geq L_1$ and $|\nabla^k v|\lesssim_k \varepsilon_2$.
We also take a cut-off $\varphi:\mathbb{R}\to [0,1]$ so that 
$\supp(d\varphi) \subset (L_1-2,L_1-1)\cup (T-L_1+1,T-L_1+2)$,
$\varphi(t)=1$ over $\{t\leq L_1-2\}\cup \{t\geq T-L_1+2\}$ and 
$\varphi=0$ over $L_1-1\leq t\leq T-L_1+1$.
We define a gauge transformation $h$ of $E$ over $-1<t<T+1$ by 
\begin{equation*}
  h = \begin{cases}
       g_A^{-1}\circ (v e^{\varphi u_1})\circ g_B \quad (t\leq T/2),\\
       g_A^{-1}\circ (v e^{\varphi u_2})\circ g_B \quad (t > T/2).
      \end{cases}
\end{equation*}
Then  
$|\nabla_A^k (h(B)-A)|\lesssim_k \exp(2|t-T/2|-T)$ over $L_1<t<T-L_1$, 
$\norm{h(B)-A}_{L^2_{10,A}(0,L_1)}\lesssim \varepsilon_2+e^{-2L_1}$ and 
$\norm{h(B)-A}_{L^2_{10,A}(T-L_1,T)}\lesssim \varepsilon_2+e^{-2L_1}$.
We can choose $L_1$ and $\varepsilon_2$ so that $h$ satisfies the statement.
\end{proof}

Using Lemmas \ref{lemma: gluing gauge transformations non-degenerate} and 
\ref{lemma: gluing gauge transformations degenerate}, we can provide a sufficient condition 
for a given connection $[B]$ to be contained in $U_\delta(A,T)$:

\begin{lemma} \label{lemma: gluing gauge transformations unify}
 For any $\delta>0$ we can choose $\varepsilon_3=\varepsilon_3(\delta)>0$ and 
an integer $R_1=R_1(\delta)>L_1(\delta)$ ($L_1(\delta)$ is the constant introduced in 
  Lemma \ref{lemma: gluing gauge transformations degenerate})
so that the following statement holds.
Take $[A],[B]\in \moduli_d$ and an integer $T>1$. If they satisfy the following two conditions, then 
$[B]\in U_\delta(A,T)$. 
  \begin{itemize}
  \item $G(A)\cap [-T-R_1,T+R_1] = G(B)\cap [-T-R_1,T+R_1]$. 
  Let $n_1<n_2<\dots<n_G$ be the elements of this set, and we set $p_k = (n_k+L_1,\theta_0)$ and 
  $q_k=(n_k-L_1+1,\theta_0)$ for $1\leq k\leq G$. 
  \item For each $1\leq k\leq G$ there exists a gauge transformation $g_k$ of $E$ over $n_k-R_1<t<n_k+R_1$
  satisfying 
  \begin{equation*}
     \begin{split}
    \norm{g_k(B)-A}_{L^2_{10,A}(n_k-R_1,n_k+R_1)}<\varepsilon_3 \quad (1\leq k\leq G),\\
    \min_{\pm} \dist_{SU(2)}(g_k(p_k),\pm g'_{k+1}(q_{k+1}))<\varepsilon_3  \quad (1\leq k<G).
     \end{split}
  \end{equation*}
  Here $g'_{k+1}(q_{k+1})$ is defined by the following commutative
   diagram. 
     \begin{equation*}
       \begin{CD}
         E_{p_k} @>{\text{parallel translation by $B$}}>> E_{q_{k+1}} \\
         @VV{g'_{k+1}(q_{k+1})}V @VV{g_{k+1}(q_{k+1})}V \\
         E_{p_k} @>{\text{parallel translation by $A$}}>> E_{q_{k+1}}
       \end{CD}
     \end{equation*}
  \end{itemize}
\end{lemma}
\begin{proof}
First let's consider the case $G(A)\cap [-T-R_1,T+R_1] = G(B)\cap [-T-R_1,T+R_1] =\emptyset$.
By Lemma \ref{lemma: exponential decay} we can choose trivializations $g_A$ and $g_B$ of $E$ over $-T-R_1<t<T+R_1$ such that 
the connection matrices $g_A(A)$ and $g_B(B)$ satisfy 
\[  |\nabla^k g_A(A)|, |\nabla^k g_B(B)| \lesssim_k e^{2(|t|-T-R_1)} \quad (|t|<T+R_1-1).\]
Then $h:=g_A^{-1}\circ g_B$ satisfies (if $R_1\gg 1$)
\[ e^{|n-G(A,T)|} \norm{h(B)-A}_{L^2_{10,A}(n,n+1)}
   \leq e^{|n-\{\pm T\}|}\norm{h(B)-A}_{L^2_{10,A}(n,n+1)} < \delta \]
for all $-T\leq n\leq T-1$.
Hence $[B]\in U_\delta(A,T)$.

Next suppose $G(A)\cap [-T-R_1,T+R_1] \neq \emptyset$.
From the compactness of $\moduli_d$ we can find $\kappa>0$ so that if $[C]\in \moduli_d$ satisfies 
$\norm{F_C}_{L^\infty(0,1)}\geq \nu$ then $\norm{F_C}_{L^\infty(n,n+1)}\geq \kappa$ for all integers
$|n| \leq L_1+1$.
Let $\varepsilon_1 = \varepsilon_1(\kappa, \delta e^{-L_1-1})$ and $\varepsilon_2 = \varepsilon_2(\delta)$ 
be the positive constants introduced in Lemmas \ref{lemma: gluing gauge transformations non-degenerate} and 
\ref{lemma: gluing gauge transformations degenerate}.
We take $\varepsilon_3>0$ and $R_1>0$ so that 
\[ \varepsilon_3 < \min(\varepsilon_1, \varepsilon_2), \quad R_1>L_1+2, \quad \varepsilon_3 e^{R_1} < \delta.\]

We inductively define gauge transformations $h_k$ of $E$ over $n_1-R_1<t<n_k+R_1$ for $k=1,2,\dots,G$ so that 
the following two conditions hold:
\begin{itemize}
  \item $h_k = g_1$ over $n_1-R_1<t<n_1$ and $h_k=\pm g_k$ over $n_k<t<n_k+R_1$.
  \item $e^{|n-G(A,T)|}\norm{h_k(B)-A}_{L^2_{10,A}(n,n+1)} <\delta$ for all integers $n_1\leq n <n_k$.  
\end{itemize}
$h_1:=g_1$ obviously satisfies the conditions. 
Suppose we have constructed $h_k$ $(k< G)$.

\textbf{Case 1}. Suppose $n_{k+1}-n_k-1<2L_1$. Set $m=\lfloor \frac{n_k+n_{k+1}}{2}\rfloor$.
From the definition of $\kappa$ we have $\norm{F_A}_{L^\infty(m,m+1)}\geq \kappa$.
We also have $\norm{h_k(B)-A}_{L^2_{10,A}(m,m+1)}, \norm{g_{k+1}(B)-A}_{L^2_{10,A}(m,m+1)}<\varepsilon_3<\varepsilon_1$.
Then we can glue $h_k$ and $g_{k+1}$ over $m<t<m+1$ by Lemma \ref{lemma: gluing gauge transformations non-degenerate}
and get $h_{k+1}$. This satisfies the required conditions.

\textbf{Case 2}. Suppose $n_{k+1}-n_k-1\geq 2L_1$. Then we can apply Lemma \ref{lemma: gluing gauge transformations degenerate}.
We glue $h_k$ and $g_{k+1}$ over $n_k+1<t<n_{k+1}$ and get $h_{k+1}$.

Therefore we get $h_G$ over $n_1-R_1<t<n_G+R_1$.
If $(-T,T)\subset (n_1-R_1,n_G+R_1)$, then it satisfies 
\begin{equation} \label{eq: required condtion for U_delta}
 e^{|n-G(A,T)|}\norm{h_G(B)-A}_{L^2_{10,A}(n,n+1)} < \delta
\end{equation}
for all integers $-T\leq n <T$.
Hence $[B]\in U_\delta(A,T)$.

So the remaining case is $(-T,T)\not\subset (n_1-R_1,n_G+R_1)$.
Suppose $-T<n_1-R_1$. Then $\norm{F_A}_{L^\infty(-T-R_1,n_1)} < \nu$ and $\norm{F_B}_{L^\infty(-T-R_1,n_1)}<\nu$.
Hence by Lemma \ref{lemma: exponential decay}
there are trivializations $g_A$ and $g_B$ of $E$ over $-T-R_1<t<n_1$ such that 
the connection matrices $g_A(A)$ and $g_B(B)$ satisfy appropriate exponential decay conditions.
We glue $g_A^{-1}\circ g_B$ to $h_G$ as in the proof of Lemma \ref{lemma: gluing gauge transformations degenerate}.
In the case of $T>n_G+R_1$, we proceed in the same way over $n_G+1 <t< T+R_1$.
Then we get a gauge transformation $h$ of $E$ over $-T<t<T$ satisfying  
(\ref{eq: required condtion for U_delta}) for all integers $-T\leq n<T$.
Thus $[B]\in U_\delta(A,T)$
\end{proof}

By using Lemma \ref{lemma: gluing gauge transformations unify}
we prove Proposition \ref{prop: decomposition}. We write the statement again for 
the convenience of readers.
\begin{proposition}[$=$ Proposition \ref{prop: decomposition}]
For any $\delta>0$ and any integer $T>1$ there exist 
$[A_1],\dots,[A_n]\in \moduli_d$ satisfying 
\[ \log n\lesssim_\delta T, \quad \moduli_d = \bigcup_{i=1}^n U_\delta(A_i,T).\] 
\end{proposition}
\begin{proof}
Let $\varepsilon_3=\varepsilon_3(\delta)$ and $R_1=R_1(\delta)$ be the positive constants introduced in 
Lemma \ref{lemma: gluing gauge transformations unify}.
Let $\varepsilon=\varepsilon(\delta)<\varepsilon_3$ be a small positive number which will be fixed later.
For each subset $\Omega\subset \mathbb{Z}\cap [-T-R_1,T+R_1]$ we define 
\[ \moduli_d^{\Omega} = \{[A]\in \moduli_d|\, G(A)\cap [-T-R_1,T+R_1] =\Omega\}.\]
$\moduli_d$ is decomposed into these $\moduli_d^\Omega$, and the number of the choices of $\Omega\subset
\mathbb{Z}\cap [-T-R_1,T+R_1]$ is equal to $2^{2(T+R_1)+1}\lesssim_\delta 4^T$.

We choose an open cover $\alpha$ of $\moduli_d$ such that if $[A],[B]\in \moduli_d$ is contained in the same 
open set $U\in \alpha$ then there exists a gauge transformation $g$ of $E$ over $-R_1<t<R_1$ satisfying 
\[ \norm{g(B)-A}_{L^2_{10,A}(-R_1,R_1)}<\varepsilon.\]
Note that the choice of $\alpha$ depends on $\delta$ and $\varepsilon$.

Take $\Omega = \{n_1<n_2<\dots<n_G\} \subset \mathbb{Z}\cap [-T-R_1,T+R_1]$.
We define an open covering $\mathcal{U}$ of $\moduli_d$ by 
\[ \mathcal{U} = \bigvee_{k=1}^G (-n_k)\cdot \alpha.\]
Here $(-n_k)\cdot \alpha$ is the translation of $\alpha$ by $(-n_k)$, and $\mathcal{U}$ is the set of 
open subsets $U_1\cap\dots \cap U_G$ $(U_k\in (-n_k)\cdot \alpha)$.
The cardinality of $\mathcal{U}$ is bounded by $|\alpha|^{G} \leq |\alpha|^{2T+2R_1+1}$.

We choose $V\in \mathcal{U}$ and consider $\moduli_d^\Omega \cap V$.
Let $\mathcal{A}$ be the set of connections $A$ on $E$ satisfying $[A]\in \moduli_d^\Omega\cap V$.
Take and fix one $A_0\in \mathcal{A}$.
For every $A\in \mathcal{A}$ and $1\leq k\leq G$ there exists a gauge transformation $g_{A,k}$ over 
$n_k-R_1<t<n_k+R_1$ satisfying 
\[ \norm{g_{A,k}(A_0)-A}_{L^2_{10,A}(n_k-R_1,n_k+R_1)} < \varepsilon.\]
Let $L_1=L_1(\delta)>0$ be the positive constant introduced in Lemma \ref{lemma: gluing gauge transformations degenerate},
and set $p_k=(n_k+L_1,\theta_0)$ and $q_k = (n_k-L_1+1,\theta_0)$ for $1\leq k\leq G$.
We consider the map:
\[ \mathcal{A}\to SU(2)^{G-1}, \quad A\mapsto (g_{A,k}(p_k)^{-1}g_{A,k+1}'(q_{k+1}))_{k=1}^{G-1}.\]
Here $g_{A,k+1}'(q_{k+1})$ is defined by the commutative diagram:
\begin{equation*}
       \begin{CD}
         E_{p_k} @>{\text{parallel translation by $A_0$}}>> E_{q_{k+1}} \\
         @VV{g'_{A,k+1}(q_{k+1})}V @VV{g_{A,k+1}(q_{k+1})}V \\
         E_{p_k} @>{\text{parallel translation by $A$}}>> E_{q_{k+1}}
       \end{CD}
     \end{equation*}
Considering a covering of $SU(2)$ by $\varepsilon$-balls, we can construct a decomposition 
$\mathcal{A}=\mathcal{A}_1\cup \dots\cup \mathcal{A}_N$ such that 
\begin{itemize}
  \item $\log N\lesssim_\varepsilon G\lesssim_\delta T$.
  \item If $A,B\in \mathcal{A}$ is contained in the same $\mathcal{A}_i$ then 
  \[ \dist_{SU(2)}(g_{A,k}(p_k)^{-1}g'_{A,k+1}(q_{k+1}),\, g_{B,k}(p_k)^{-1}g'_{B,k+1}(q_{k+1})) < \varepsilon 
    \quad (\forall 1\leq k\leq G).\]
\end{itemize}

\begin{claim}
For any $1\leq i\leq G$ and $A,B\in \mathcal{A}_i$ we get $[B]\in U_\delta(A,T)$.
\end{claim}
\begin{proof}
We check the conditions of Lemma \ref{lemma: gluing gauge transformations unify}.
The condition $G(A)\cap [-T-R_1,T+R_1]=G(B)\cap [-T-R_1,T+R_1]$ is satisfied.
For each $1\leq k\leq G$ we set 
$g_k = g_{A,k}\circ g_{B,k}^{-1}$ over $n_k-R_1<t<n_k+R_1$. They satisfy 
\[ \dist_{SU(2)}(g_k(p_k), g'_{k+1}(q_{k+1}))<\varepsilon <\varepsilon_3 \quad (\forall 1\leq k\leq G-1).\]
We have 
\[ g_k(B)-A = g_k(B-g_{B,k}(A_0)) + g_{A,k}(A_0)-A.\]
Hence we can choose $\varepsilon = \varepsilon(\delta)>0$ so small that 
\[ \norm{g_k(B)-A}_{L^2_{10,A}(n_k-R_1,n_k+R_1)} < \varepsilon_3.\]
Then we can apply Lemma \ref{lemma: gluing gauge transformations unify} to $A$ and $B$, and we get 
$[B]\in U_\delta(A,T)$.
\end{proof}

Pick up $A_1\in \mathcal{A}_1, \dots, A_N\in \mathcal{A}_N$. Then by the above claim 
\[ \moduli_d^\Omega\cap V \subset U_\delta(A_1,T)\cup \dots \cup U_\delta(A_N,T).\]
We have the following bounds on several parameters:
$\log N\lesssim_\delta T$.
The number of the choices of $V\in \mathcal{U}$ is $\lesssim_\delta |\alpha|^{2T}$.
Note that $|\alpha|$ is now a constant depending only on $\delta$.
The number of the choices of $\Omega\subset \mathbb{Z}\cap [-T-R_1,T+R_1]$ is 
$\lesssim_\delta 4^T$.
Combining these estimates, we get the conclusion.
\end{proof}

\section{Instanton approximation: Proof of Proposition \ref{prop: instanton approximation}}
\label{section: instanton approximation}

We develop instanton approximation technique and
prove Proposition \ref{prop: instanton approximation} in this section.
First we prepare some facts concerning a Green kernel function.
Let $\Delta = \nabla^*\nabla$ be the Laplacian on functions in $X$.
Our sign convention of $\Delta$ is geometric (
$\Delta=-\partial^2/\partial x_1^2 -\partial^2/\partial x_2^2-\partial^2/\partial x_3^2-\partial^2/\partial x_4^2$
over $\mathbb{R}^4$).
Let $g(x,y)$ be the Green kernel of $\Delta +2$ over $X$.
This satisfies 
\[ (\Delta_y + 2)g(x,y) = \delta_x(y) \]
in the distributional sense, i.e. for any compactly supported smooth function $\varphi$ over $X$
\[ \varphi(x) = \int_X g(x,y)(\Delta_y+2)\varphi(y)d\vol(y).\]
$g(x,y)$ is positive everywhere.
It is smooth outside the diagonal, and its singularity along the diagonal is $\dist(x,y)^{-2}$:
\[ \dist(x,y)^{-2} \lesssim g(x,y) \lesssim \dist(x,y)^{-2} \quad (\dist(x,y)\leq 1).\]
It decays exponentially in a long range:
\begin{equation} \label{eq: exponential decay of Green kernel}
 g(x,y) \lesssim e^{-\sqrt{2}\, \dist(x,y)}  \quad (\dist(x,y)>1).
\end{equation}
A detailed construction of $g(x,y)$ is explained in \cite[Appendix]{Matsuo--Tsukamoto}.

For $u\in \Omega^i(\ad E)$ we define its \textbf{Taubes norm} $\tnorm{u}$ by 
\[ \tnorm{u} = \sup_{x\in X} \int_X g(x,y)|u(y)|d\vol(y).\]
This was introduced by Taubes \cite{Taubes path} and Donaldson \cite{Donaldson approximation}. 
An importance of this norm is linked to the following Weitzenb\"{o}ck formula.
Let $A$ be a connection on $E$.
For $\phi\in \Omega^+(\ad E)$ we have (\cite[Chapter 6]{Freed--Uhlenbeck}):
\[ d_A^+d_A^*\phi = \frac{1}{2}\nabla_A^*\nabla_A \phi + \left(\frac{S}{6}- W^+\right)\phi + F_A^+\cdot\phi ,\]
where $S$ is the scalar curvature of $X$ and $W^+$ is the self-dual part of the Weyl curvature.
Since $X = \mathbb{R}\times S^3$ is conformally flat, we have $W^+=0$.
The scalar curvature $S$ is constantly equal to $6$.
So we get 
\begin{equation} \label{eq: Weitzenbock formula}
 d_A^+d_A^*\phi = \frac{1}{2}(\nabla_A^*\nabla_A + 2) \phi + F_A^+\cdot\phi.
\end{equation}
For any smooth $\eta \in \Omega^+(\ad E)$ with $\norm{\eta}_{L^\infty(X)}<\infty$ there uniquely exists 
smooth $\phi\in \Omega^+(\ad E)$ satisfying 
\[ (\nabla^*_A\nabla_A+2)\phi=\eta, \quad \norm{\phi}_{L^\infty(X)}<\infty.\]
We sometimes denote $\phi$ by $(\nabla_A^*\nabla_A+2)^{-1}\eta$.
It satisfies 
\begin{equation} \label{eq: Green kernel estimate}
 |\phi(x)| \leq \int_X g(x,y)|\eta(y)|d\vol(y), \quad 
    \norm{\phi}_{L^\infty(X)} \leq \tnorm{\eta}.
\end{equation}
Moreover it satisfies the following. (Indeed this is the most spectacular property of the Taubes norm).
\begin{equation} \label{eq: quadratic Taubes estimate}
  \tnorm{(d_A^*\phi\wedge d_A^*\phi)^+} \leq 10 \tnorm{\eta}^2.
\end{equation}
For the detailed proofs of the above estimates, see \cite[Section 4, Appendix]{Matsuo--Tsukamoto}.

We define $\mathcal{A}$ as the set of connections $A$ on $E$ such that
\[ \text{$F_A^+$ is compactly supported}, \quad \tnorm{F_A^+}\leq \frac{1}{1000},\quad  
   \norm{F_A}_{C^5_A}:= \max_{0\leq k\leq 5}\norm{\nabla_A^k F_A}_{L^\infty(X)} < \infty.\]
Here $1/1000$ has no special meaning. Any sufficiently small number will do. 
The last condition is connected to the following fact:
Take any point $p\in X$. Let $g$ be the exponential gauge of radius $\pi/2$ around $p$. (The injectivity 
radius of $X$ is equal to $\pi$.)
Then the connection matrix $g(A)$ satisfies 
\[ |\nabla^k g(A)| \lesssim \norm{F_A}_{C^k_A}.\]
We summarize the results of \cite[Sections 4 and 5]{Matsuo--Tsukamoto} in the following proposition.

\begin{proposition} \label{prop: instanton approximation preliminary}
We can construct a gauge equivariant map 
\[ \mathcal{A}\ni A\mapsto \phi_A\in \Omega^+(\ad E)\]
satisfying the following conditions.

\noindent 
(1) $A+d_A^*\phi_A$ is an ASD connection.

\noindent 
(2) $\phi_A$ is smooth and 
\[ |\phi_A(x)|\lesssim \int_X g(x,y)|F_A^+(y)|d\vol(y), \quad \norm{\phi_A}_{L^\infty(X)}\lesssim \tnorm{F_A^+},
   \quad \norm{\nabla_A\phi_A}_{L^\infty(X)} <\infty.\]

\noindent 
(3) If $F_A$ is compactly supported, then 
\[ \int_X |F(A+d_A^* \phi_A)|^2d\vol = \int_X \mathrm{tr}(F_A^2).\]

\noindent 
(4) For any $A,B\in \mathcal{A}$, $\norm{\phi_A-\phi_B}_{L^\infty(X)} \lesssim \norm{A-B}_{C^1_A}$.
\end{proposition}

\begin{proof}
We roughly explain the construction of $\phi_A$ for the convenience of readers. 
Let $\Omega^+(\ad E)_0$ be the set of smooth $\eta\in \Omega^+(\ad E)$ satisfying 
$\lim_{x\to \pm\infty} |\eta(x)|=0$.
Take $\eta\in \Omega^+(\ad E)_0$ and set $\phi=(\nabla_A^*\nabla_A+2)^{-1}\eta \in \Omega^+(\ad E)_0$.
We want to solve the equation $F^+(A+d_A^*\phi)=0$. This is equivalent to 
\[ \eta = -2F_A-2F_A^+\cdot \phi -2(d_A^*\phi\wedge d_A^*\phi)^+.\]
We denote the right-hand-side by $\Phi(\eta)$. 
By using the estimates (\ref{eq: Green kernel estimate}) and (\ref{eq: quadratic Taubes estimate}), we can prove that 
$\Phi$ becomes a contraction map with respect to the Taubes norm over 
\[ \left\{\eta\in \Omega^+(\ad E)_0|\, \tnorm{\eta} \leq \frac{3}{1000}\right\}.\]
Therefore the sequence $\eta_n$ defined by 
\[ \eta_0=0, \quad \eta_{n+1}= \Phi(\eta_n) \]
is a Cauchy sequence with respect to the Taubes norm.
Then $\phi_n := (\nabla_A^*\nabla_A+2)^{-1}\eta_n$ is a convergent sequence in $L^\infty(X)$.
Let $\phi_A$ be the limit of $\phi_n$.
We can prove that $\phi_A$ is smooth and $\phi_n$ converges to $\phi_A$ in $C^\infty$ over every compact subset of $X$.
Then it satisfies $F^+(A+d_A^* \phi_A)=0$.
The conditions (2), (3) and (4) can be checked by a detailed investigation of the above construction.
\end{proof}

We need some more detailed estimates on $\phi_A$. They are established in the next two lemmas.

\begin{lemma} \label{lemma: gradient estimate}
We can choose $0<\tau<1/1000$ so that the following statement holds.
If $A\in \mathcal{A}$ satisfies $\tnorm{F^+_A}\leq \tau$ then $\phi_A$ satisfies
\[ \norm{\nabla_A\phi_A}_{L^\infty(X)} \leq 1+\norm{F_A}_{C^1_A}.\]
\end{lemma}
\begin{proof}
Suppose the statement is false.
Then for any $n>0$ there exists $A_n\in \mathcal{A}$ such that $\tnorm{F^+(A_n)}\leq 1/n$ and 
\[ R_n := \norm{\nabla_{A_n}\phi_{A_n}}_{L^\infty(X)} > 1+ \norm{F(A_n)}_{C^1_{A_n}}.\]
Take $p_n\in X$ satisfying $|\nabla_{A_n}\phi_{A_n}(p_n)|>R_n/2$.
We consider the geodesic coordinate and the exponential gauge (w.r.t. $A_n$) of radius $\pi/2$ around $p_n$.
Then the connection matrix of $A_n$ in this gauge (also denoted by $A_n$) satisfies 
\[ |A_n|+|\nabla A_n|\lesssim \norm{F(A_n)}_{C^1_{A_n}} < R_n.\]
We have the ASD equation 
\[ (\nabla_{A_n}^*\nabla_{A_n}+2)\phi_{A_n} 
   = -2F^+(A_n) - 2F^+(A_n)\cdot \phi_{A_n} - 2(d_{A_n}^*\phi_{A_n}\wedge d_{A_n}^*\phi_{A_n})^+\]
and the estimates $\norm{\phi_{A_n}}_{L^\infty}\lesssim \tnorm{F^+(A_n)} \leq 1/n$ and $\norm{F^+(A_n)}_{L^\infty}<R_n$.
Then
\[ \left|\sum_{i,j}g^{ij}(x)\partial_i \partial_j \phi_{A_n}\right|\lesssim R_n^2 \quad (|x|\leq \pi/2).\]
Here $x$ is the geodesic coordinate around $p_n$.
Set $\phi_n(y)=\phi_{A_n}(y/R_n)$ for $|y|\leq \pi/2$.
This satisfies 
\[ |\nabla \phi_n(0)|>1/2, \quad \left|\sum_{i,j}g^{ij}(y/R_n)\partial_i \partial_j \phi_n\right| \lesssim 1.\]
From the latter condition and $\norm{\phi_n}_{L^\infty}\lesssim 1/n$, 
$\phi_n$ converges to $0$ in $C^1$ over $|y|\leq \pi/3$.
But this contradicts $|\nabla \phi_n(0)|>1/2$.
\end{proof}

For $T>1$ and $K>0$ we define $\mathcal{A}(T,K)\subset \mathcal{A}$ as the set of connections $A$ on $E$ satisfying 
\[ \tnorm{F_A^+} \leq \tau, \quad \supp(F_A^+)\subset \{(t,\theta)\in \mathbb{R}\times S^3|\, T-1<|t|<T\}, \quad 
   \norm{F_A}_{C^5_A} \leq K.\]
Here $\tau$ is the positive constant introduced in Lemma \ref{lemma: gradient estimate}.
For $x=(t,\theta)\in \mathbb{R}\times S^3$ we set
\begin{equation*}
   \begin{split}
    &g_T(x) = g_T(t) = e^{-\sqrt{2}|t-T|} + e^{-\sqrt{2}|t+T|}, \\
    &\hat{g}_T(x)=\hat{g}_T(t) = (1+|t-T|)e^{-\sqrt{2}|t-T|}+(1+|t+T|)e^{-\sqrt{2}|t+T|}.
   \end{split}
\end{equation*}
From the exponential decay estimate (\ref{eq: exponential decay of Green kernel}), the Green kernel $g(x,y)$ satisfies 
\[ \int_{T-1<|t|<T} g(x,y)d\vol(y) \lesssim g_T(x), \quad 
   \int_X g(x,y)g_T(y)d\vol(y) \lesssim \hat{g}_T(x).\]

\begin{lemma} \label{lemma: continuity of instanton approximation}
(1) For any $A\in \mathcal{A}(T,K)$ and $0\leq k\leq 5$, $|\nabla_A^k \phi_A(x)|\lesssim_K g_T(x)$.

\noindent 
(2) There exists $L_2=L_2(K)>1$ such that every $A\in \mathcal{A}(T,K)$ satisfies 
\begin{equation*}
  \begin{split}
   &\left|\int_{T-L_2<t<T+L_2} |F(A+d_A^*\phi_A)|^2d\vol -\int_{T-L_2<t<T+L_2}\mathrm{tr}(F_A^2)\right| \leq 1/10, \\
   &\left|\int_{-T-L_2<t<-T+L_2} |F(A+d_A^*\phi_A)|^2d\vol -\int_{-T-L_2<t<-T+L_2}\mathrm{tr}(F_A^2)\right| \leq 1/10.
  \end{split}
\end{equation*}

\noindent 
(3) For any $A,B\in \mathcal{A}(T,K)$ and $0\leq k\leq 5$ 
\[  |\nabla_A^k\phi_A(x) - \nabla_B^k\phi_B(x)|\lesssim_K  \hat{g}_T(x) \norm{A-B}_{C^5_A}.\]
\end{lemma}

\begin{proof}
(1) From Proposition \ref{prop: instanton approximation preliminary} (2), 
$|\phi_A(x)|\lesssim_K g_T(x)$. 
By Lemma \ref{lemma: gradient estimate}, $\norm{\nabla_A\phi_A}_{L^\infty} \lesssim_K 1$.
Set $R =\sup_{t\in \mathbb{R}}g_T(t)^{-1}\norm{\phi_A}_{L^2_{2,A}(t,t+1)}$.
We have the ASD equation
\[ (\nabla_A^*\nabla_A+2)\phi_A = -2F_A^+ -2F_A^+\cdot \phi_A -2(d_A^*\phi_A\wedge d_A^*\phi_A)^+.\]
From the elliptic estimate 
\begin{equation*}
   \begin{split}
   \norm{\phi_A}_{L^2_{2,A}(t,t+1)}&\lesssim_K \norm{\phi_A}_{L^2(t-1,t+2)}+\norm{(\nabla_A^*\nabla_A+2)\phi_A}_{L^2(t-1,t+2)}\\
   &\lesssim_K g_T(t) + \norm{d_A^*\phi_A\wedge d_A^*\phi_A}_{L^2(t-1,t+2)} \\
   &\lesssim_K g_T(t) + \norm{d_A^*\phi_A}_{L^2(t-1,t+2)} \quad (\norm{\nabla_A\phi_A}_{L^\infty}\lesssim_K 1).
   \end{split}
\end{equation*}  
Let $\varepsilon = \varepsilon(K)>0$ be a small number which will be fixed later.
From the interpolation (Gilbarg--Trudinger \cite[Theorem 7.28]{Gilbarg--Trudinger}),
\[ \norm{d_A^*\phi_A}_{L^2(t-1,t+2)}\leq 
   C(\varepsilon, K)\norm{\phi_A}_{L^2(t-1,t+2)} + \varepsilon \norm{\phi_A}_{L^2_{2,A}(t-1,t+2)}.\]
Hence 
\[ \norm{\phi_A}_{L^2_{2,A}(t,t+1)}\leq C'(\varepsilon,K)g_T(x) + C''(K)\varepsilon \norm{\phi_A}_{L^2_{2,A}(t-1,t+2)}.\]
Then 
\[ R\leq C'(\varepsilon, K) + C'''(K)\varepsilon R.\]
We choose $\varepsilon$ so that $C'''(K)\varepsilon < 1/2$. Then $R\lesssim_K 1$, i.e.
$\norm{\phi_A}_{L^2_{2,A}(t,t+1)} \lesssim_K g_T(x)$.
The rest of the argument is a standard bootstrapping.

(2) Set $a=d_A^*\phi_A$ and $cs_A(a) = \mathrm{tr}(2a\wedge F_A + a\wedge d_A a + \frac{2}{3}a^3)$.
We have $\mathrm{tr}(F(A+a)^2) -\mathrm{tr}F_A^2 = d cs_A(a)$. Then by the Stokes theorem 
\[ \int_{T-L_2<t<T+L_2}|F(A+a)|^2d\vol - \int_{T-L_2<t<T+L_2}\mathrm{tr} F_A^2 = 
   \int_{t=T+L_2}cs_A(a) - \int_{t=T-L_2}cs_A(a).\]
By (1), the right-hand-side goes to zero (uniformly in $A$ and $T$) as $L_2 \to \infty$.

(3) From (1), $|\nabla_A^k \phi_A(x)|, |\nabla_B^k \phi_B(x)| \lesssim_K g_T(x)$ for $0\leq k\leq 5$.
Set $a=B-A$. It is enough to prove the statement under the assumption $\norm{a}_{C^5_A}< 1$.
From the ASD equation, 
\begin{equation} \label{eq: equation for comparing A and B}
   \begin{split}
    &(\nabla_A^* \nabla_A+2)(\phi_A-\phi_B) = 2(F_B^+ -F_A^+) + 2(F_B^+\cdot\phi_B - F_A^+\cdot\phi_A)\\
    &+ 2\left\{(d_B^*\phi_B\wedge d_B^*\phi_B)^+ -(d_A\phi_A^*\wedge d_A^*\phi_A)^+\right\} 
    + a*\nabla_B\phi_B + (\nabla_A a)*\phi_B + a*a*\phi_B.
   \end{split}
\end{equation}
For any $t\in \mathbb{R}$, by the elliptic estimate
\begin{equation} \label{eq: equation comparing A and B second}
   \norm{\phi_A-\phi_B}_{L^2_{2,A}(t,t+1)} \lesssim_K \norm{\phi_A-\phi_B}_{L^2(t-1,t+2)} + g_T(t)\norm{a}_{C^1_A}
    + g_T(t) \norm{d_A^*\phi_A-d_A^*\phi_B}_{L^2(t-1,t+2)}.
\end{equation}
From Proposition \ref{prop: instanton approximation preliminary} (4) we have $\norm{\phi_A-\phi_B}_{L^\infty}\lesssim \norm{a}_{C^1_A}$.
So we get 
\[ \norm{\phi_A-\phi_B}_{L^2_{2,A}(t,t+1)}  \lesssim_K \norm{a}_{C^1_A} + \norm{d_A^*\phi_A-d_A^*\phi_B}_{L^2(t-1,t+2)}.\]
By using the interpolation as in (1), we get 
\[ \norm{\phi_A-\phi_B}_{L^2_{2,A}(t,t+1)} \lesssim_K \norm{a}_{C^1_A}.\]
Then the bootstrapping shows 
$\norm{\phi_A-\phi_B}_{C^1_A} \lesssim_K \norm{a}_{C^1_A}$.
By this estimate, the modulus of the right-hand-side of (\ref{eq: equation for comparing A and B})
is $\lesssim_K g_T(x) \norm{a}_{C^1_A}$.
Then by the Green kernel estimate (\ref{eq: Green kernel estimate})
\[ |\phi_A(x)-\phi_B(x)|\lesssim_K \hat{g}_T(x)\norm{a}_{C^1_A}.\]
Using this and $\norm{\phi_A-\phi_B}_{C^1_A} \lesssim_K \norm{a}_{C^1_A}$ in (\ref{eq: equation comparing A and B second}), 
we get $\norm{\phi_A-\phi_B}_{L^2_{2,A}(t,t+1)} \lesssim_K \hat{g}_T(t)\norm{a}_{C^1_A}$.
The rest of the proof is a bootstrapping.
\end{proof}

The next lemma is a preliminary version of Proposition \ref{prop: instanton approximation}.
Here we connect the set $U_\delta(A,T)$ to $\mathcal{A}(T,K)$ above.

\begin{lemma} \label{lemma: gluing instantons}
There exist positive numbers $\delta_1$ and $K$ such that 
for any $[A]\in \moduli_d$, any integer $T>1$ and $0<\delta\leq \delta_1$ we can construct a 
(not necessarily continuous) map
\[ U_\delta(A,T)\to \mathcal{A}(T,K), \quad [B]\mapsto \hat{B},\]
satisfying the following conditions.

\noindent 
(1) There exists a gauge transformation $g$ of $E$ over $|t|<T-1$ satisfying $g(\hat{B})=B$.

\noindent 
(2) There exists a gauge transformation $h$ of $E$ satisfying 
\[ \sup_{n\in \mathbb{Z}} e^{|n-G(\hat{A})|}\norm{h(\hat{B})-\hat{A}}_{L^2_{10,\hat{A}}(n,n+1)}\lesssim \delta.\]

\noindent
(3) The curvature $F(\hat{A})$ is supported in $|t|<T$. Moreover 
\begin{equation*}
  \begin{split}
    &\left|\int_X \mathrm{tr}(F(\hat{A})^2) - \int_{-T<t<T}|F_A|^2d\vol\right| \lesssim 1,  \\
    \int_{T-1<t<T}&\mathrm{tr}(F(\hat{A})^2) \geq 10, \quad  \int_{-T<t<-T+1}\mathrm{tr}(F(\hat{A})^2) \geq 10. 
  \end{split}
\end{equation*}
\end{lemma}

\begin{proof}
Choose a representative $A$ of $[A]$. First we define $\hat{A}$.
We take a cut-off $\varphi:\mathbb{R}\to [0,1]$ such that $\supp(d\varphi)\subset \{T-1/2<|t|<T\}$, 
$\varphi=1$ over $|t|\leq T-1/2$ and $\varphi=0$ over $|t|\geq T$.
We can choose a trivialization $u$ of $E$ over $T-1<|t|<T$ so that the connection matrix $u(A)$ satisfies 
$\norm{u(A)}_{C^{10}}\lesssim 1$.
We define a connection $A_0$ by $A_0 = u^{-1}(\varphi u(A))$.
$A_0=A$ over $|t|\leq T-1/2$, and $A_0$ is flat over $|t|\geq T$.
The self-dual curvature $F^+(A_0)$ is supported in $T-1/2<|t|<T$.
We try to reduce its Taubes norm by 
gluing sufficiently many concentrated instantons to $A_0$ over $T-1/2<|t|<T$.
This is a rather standard technique for specialists of gauge theory. For the detail,
see Donaldson \cite[pp. 190-199]{Donaldson approximation}.
After this gluing procedure, we get a connection $\hat{A}$ such that 
$\hat{A}=A$ over $|t|\leq T-1/2$, $F(\hat{A})$ is supported in $|t|<T$ and
\[ \supp(F^+(\hat{A}))\subset \{T-1/2<|t|<T\}, \quad 
   \tnorm{F^+(\hat{A})} \leq \tau/2, \quad 
   \norm{F(\hat{A})}_{C^{5}_{\hat{A}}} \lesssim 1.\]
We can also assume that $\hat{A}$ satisfies the condition (3) of the statement.
The last condition of (3) can be achieved by increasing the number of gluing instantons.
Moreover, by the same reasoning, we can assume $\norm{F(\hat{A})}_{L^\infty(T-1,T)}, \norm{F(\hat{A})}_{L^\infty(-T,-T+1)}\geq \nu$.
Hence $-T, T-1\in G(\hat{A})$. This fact together with $\hat{A}=A$ over $|t|\leq T-1/2$ implies 
\begin{equation} \label{eq: weight decrease under gluing}
   |n-G(\hat{A})| \leq |n-G(A,T)| \quad (-T\leq n\leq T-1).
\end{equation}

Next we take $[B]\in U_\delta(A,T)$ ($\delta\leq \delta_1$) different from $[A]$.
We can choose a representative $B$ of $[B]$ satisfying 
\[ e^{|n-G(A,T)|} \norm{B-A}_{L^2_{10,A}(n,n+1)} \leq \delta \quad (-T\leq n\leq T-1).\]
We take a cut-off $\psi:\mathbb{R}\to [0,1]$ such that 
$\supp(d\psi)\subset \{T-1<|t|<T-1/2\}$, $\psi=1$ over $|t|\leq T-1$ and $\psi=0$ over $|t|\geq T-1/2$.
Set $\hat{B}=\psi B + (1-\psi)\hat{A}$.
This satisfies the condition (1) because $\hat{B}=B$ over $|t|\leq T-1$.
$F^+(\hat{B})$ is supported in $\{T-1<|t|<T\}$ and 
\[ \tnorm{F^+(\hat{B})} \leq \const\cdot \delta_1 + \tnorm{F^+(\hat{A})} \leq \tau\]
if we choose $\delta_1$ sufficiently small.
We can find a universal constant $K>0$ so that $\norm{F(\hat{B})}_{C^{5}_{\hat{B}}}\leq K$ for all 
$[B]\in U_{\delta_1}(A,T)$.
Then $\hat{B}\in \mathcal{A}(T,K)$.

We want to check the condition (2). $\hat{B}-\hat{A}=0$ over $|t|\geq T-1/2$. 
For $|t|<T-1/2$ we have $\hat{A}=A$ and $\hat{B}-\hat{A} = \psi(B-A)$.
Using (\ref{eq: weight decrease under gluing}), for $-T\leq n\leq T-1$
\[ e^{|n-G(\hat{A})|}\norm{\hat{B}-\hat{A}}_{L^2_{10,\hat{A}}(n,n+1)} 
   \leq e^{|n-G(A,T)|}\norm{\psi(B-A)}_{L^2_{10,A}(n,n+1)} \lesssim \delta.\]
If $n<-T$ or $n >T-1$ then $e^{|n-G(\hat{A})|}\norm{\hat{B}-\hat{A}}_{L^2_{10,\hat{A}}(n,n+1)}$ is zero.
This shows (2).
\end{proof}

Then we can prove the main result of this section.

\begin{proposition}[$=$ Proposition \ref{prop: instanton approximation}]
For any $r>0$ we can choose $\delta_0=\delta_0(r)>0$ satisfying the following statement.
For any $[A]\in \moduli_d$ and any integer $T>1$ there exists a non-flat instanton $A'$ on $E$ and 
a (not necessarily continuous) map 
\[ U_{\delta_0}(A,T)\to V_r(A'), \quad [B]\mapsto [B'] \]
such that 

\noindent 
(1) 
\[ \norm{F_{A'}}_{L^\infty(X)} \leq D_0, \quad 
  \left| \int_X |F_{A'}|^2d\vol -\int_{(-T,T)\times S^3} |F_A|^2d\vol \right| \lesssim 1.\]
Here $D_0$ is a universal constant independent of $r$.

\noindent
(2) For any $[B]\in U_{\delta_0}(A,T)$ there exists a gauge transformation $h$ of $E$ over $|t|<T-1$ satisfying
\[ |h(B')-B|\lesssim g_T(t) \quad (|t|<T-1).\]
\end{proposition}

\begin{proof}
Let $0<\delta_0 =\delta_0(r)\leq \delta_1$ ($\delta_1$ is the positive constant introduced in Lemma 
\ref{lemma: gluing instantons}). 
$\delta_0$ will be fixed later.
Take $[B]\in U_{\delta_0}(A,T)$ and set $B'= \hat{B}+d_{\hat{B}}^*\phi_{\hat{B}}$.
Here $\hat{B}$ is constructed by Lemma \ref{lemma: gluing instantons}, and $\phi_{\hat{B}}$ is constructed by 
Proposition \ref{prop: instanton approximation preliminary}.
$B'$ is an ASD connection.
$F(\hat{A})$ is compactly supported, and hence Proposition \ref{prop: instanton approximation preliminary} 
(3) implies 
\begin{equation} \label{eq: B' is instanton}
   \int_X |F(A')|^2 d\vol = \int_X \mathrm{tr}(F(\hat{A})^2) < \infty.
\end{equation}
Thus $A'$ is an instanton.
We will show $[B']\in V_r(A')$ and the above conditions (1) and (2).

First we check (1).
We have $F(A') = F(\hat{A}) + d_{\hat{A}}d_{\hat{A}}^*\phi_{\hat{A}} + (d_{\hat{A}}\phi_{\hat{A}})^2$.
Since $\hat{A}\in \mathcal{A}(T,K)$, we get $\norm{F(A')}_{L^\infty(X)}\lesssim 1$ by 
Lemma \ref{lemma: continuity of instanton approximation} (1).
Moreover by (\ref{eq: B' is instanton}) and Lemma \ref{lemma: gluing instantons} (3)
\[ \left|\int_X |F_{A'}|^2d\vol - \int_{-T<t<T}|F_A|^2 d\vol \right|
  = \left|\int_X \mathrm{tr}(F(\hat{A})^2) - \int_{-T<t<T}|F_A|^2 d\vol\right| \lesssim 1 \] 
Thus we have proved the condition (1).

Next we check (2).
From Lemma \ref{lemma: gluing instantons} (1)
we can assume $\hat{B}=B$ over $|t|<T-1$.
Then $B'-B= d_{\hat{B}}^*\phi_{\hat{B}}$ over $|t|<T-1$.
By Lemma \ref{lemma: continuity of instanton approximation} (1) we have 
$|d^*_{\hat{B}}\phi_{\hat{B}}|\lesssim g_T(t)$.
Thus $|B'-B|\lesssim g_T(t)$ over $|t|<T-1$.
This shows the condition (2).

The rest of the task is to show that $A'$ is non-flat and $[B']\in V_r(A')$.
From lemma \ref{lemma: continuity of instanton approximation} (2) and Lemma \ref{lemma: gluing instantons} (3)
\[ \int_{T-L_2<t<T+L_2}|F(A')|^2d\vol > 9, \quad \int_{-T-L_2<t<-T+L_2}|F(A')|^2d\vol >9.\]
This implies that $A'$ is not flat.
Moreover by Lemma \ref{lemma: exponential decay} the $L^\infty$-norms of $F(A')$ over 
$T-L_2<t<T+L_2$ and $-T-L_2<t<-T+L_2$ are both bounded from below by $\nu$.
Hence 
\begin{equation} \label{eq: G'(A') around T and -T}
  G'(A')\cap [T-L_2,T+L_2]\neq \emptyset, \quad  G'(A')\cap [-T-L_2,-T+L_2]\neq \emptyset.
\end{equation}
From Lemma \ref{lemma: continuity of instanton approximation} (1) $A'=\hat{A}+d_{\hat{A}}^*\phi_{\hat{A}}$ satisfies 
\[ |F(A')-F(\hat{A})|\lesssim g_T(t).\]
Then we can find a universal constant $L>L_2$ so that 
\[ t\in G(\hat{A}) \Longrightarrow  (t-L,t+L)\cap G'(A') \neq \emptyset.  \]
Then for all $n\in \mathbb{Z}$
\begin{equation} \label{eq: weights for A' and A hat}
   |n-G'(A')| \leq |n-G(\hat{A})| + L.
\end{equation}

From Lemma \ref{lemma: gluing instantons} (2) we can assume 
\begin{equation} \label{eq: difference between B hat and A hat}
   \sup_{n\in \mathbb{Z}} e^{|n-G(\hat{A})|} \norm{\hat{B}-\hat{A}}_{L^2_{10,\hat{A}}(n,n+1)} \lesssim \delta_0.
\end{equation}
$B'-A'= \hat{B}-\hat{A}+d_{\hat{B}}^*\phi_{\hat{B}}-d_{\hat{A}}^*\phi_{\hat{A}}$.
From Lemma \ref{lemma: continuity of instanton approximation} (1) 
we have $\norm{A'-\hat{A}}_{C^{4}_{\hat{A}}} \lesssim 1$.
Then 
\begin{equation*}
  \begin{split}
   e^{|n-G'(A')|}\norm{B'-A'}_{L^2_{2,A'}(n,n+1)} \lesssim \,
   &e^{|n-G'(A')|}\norm{\hat{B}-\hat{A}}_{L^2_{2,\hat{A}}(n,n+1)} \\
   &+ e^{|n-G'(A')|}\norm{d_{\hat{B}}^*\phi_{\hat{B}}-d_{\hat{A}}^*\phi_{\hat{A}}}_{L^2_{2,\hat{A}}(n,n+1)}.
  \end{split}
\end{equation*}
From (\ref{eq: weights for A' and A hat}) and (\ref{eq: difference between B hat and A hat}) 
\[ e^{|n-G'(A')|}\norm{\hat{B}-\hat{A}}_{L^2_{2,\hat{A}}(n,n+1)}  \lesssim \delta_0.\]
From Lemma \ref{lemma: continuity of instanton approximation} (3),
$\norm{d_{\hat{B}}^*\phi_{\hat{B}}-d_{\hat{A}}^*\phi_{\hat{A}}}_{L^2_{2,\hat{A}}(n,n+1)} \lesssim
\hat{g}_T(n) \norm{\hat{B}-\hat{A}}_{C^5_{\hat{A}}}$.
By (\ref{eq: difference between B hat and A hat}) and the Sobolev embedding, 
\[ e^{|n-G'(A')|}\norm{d_{\hat{B}}^*\phi_{\hat{B}}-d_{\hat{A}}^*\phi_{\hat{A}}}_{L^2_{2,\hat{A}}(n,n+1)}
   \lesssim e^{|n-G'(A')|}\hat{g}_T(n) \delta_0.\]
Recall $\hat{g}_T(n) = (1+|n-T|)e^{-\sqrt{2}|n-T|}+(1+|n+T|)e^{-\sqrt{2}|n+T|}$ and (\ref{eq: G'(A') around T and -T}).
So $e^{|n-G'(A')|}\hat{g}_T(n) \lesssim e^{|n-\{\pm T\}|}\hat{g}_T(n) \lesssim 1$.
Combining the above estimates, we conclude 
\[ \sup_{n\in \mathbb{Z}} e^{|n-G'(A')|}\norm{B'-A'}_{L^2_{2,A'}(n,n+1)} \lesssim \delta_0.\]
Recall the definition of the norm $\nnorm{\cdot}_{2,A'}$ in (\ref{eq: defintion of weighted norm}).
It uses the weight function $W_{A'}$, and this satisfies $W_{A'}(t) \lesssim e^{\alpha |t-G'(A')|}$.
Since $\alpha<1$ we get 
\[ \nnorm{B'-A'}_{2,A'} \lesssim \sup_{n\in \mathbb{Z}} e^{|n-G'(A')|}\norm{B'-A'}_{L^2_{2,A'}(n,n+1)} \lesssim \delta_0.\]
Thus we can choose $\delta_0\ll r$ so that $\nnorm{B'-A'}_{2,A'}\leq r$ and hence
 $[B']\in V_r(A')$.
\end{proof}

\section{Quantitative deformation theory: proof of Proposition \ref{prop: quantitative deformation theory}}
\label{section: quantitative deformation theory}

The purpose of this section is to prove Proposition \ref{prop: quantitative deformation theory}.
Let $D>0$ be a positive number, and let $A$ be a non-flat instanton
on $E$ satisfying $\norm{F_A}_{L^\infty(X)}\leq D$.
First we recall some notations. We denote 
\[ G'(A) = \{n\in \mathbb{Z}| \norm{F_A}_{L^\infty(n,n+1)}\geq \nu/2\}
   = \{n_1<n_2<\dots<n_G\}.\]
Let $W_A$ be the weight function introduced in Section \ref{subsection: setting}.
It is a smoothing of the function $e^{\alpha|t-G'(A)|}$ $(0<\alpha<1)$.
For $u\in \Omega^i(\ad E)$ we define ($n_0=-\infty$ and $n_{G+1}=+\infty$)
\[ \nnorm{u}_{k,A} = \max_{0\leq j\leq G}\norm{W_A u}_{L^2_{k,A}(n_j,n_{j+1})}.\]
The connection $A$ is fixed throughout this section.
So we usually abbreviate $\nnorm{u}_{k,A}$ and $\nnorm{u}_{0,A}$ to 
$\nnorm{u}_k$ and $\nnorm{u}$ respectively.
We also abbreviate the weight function $W_A$ to $W$.
We define $L^{2,W}_{k}(\Omega^i(\ad E))$ as the Banach space of locally $L^2_k$ sections
$u\in \Omega^i(\ad E)$ satisfying $\nnorm{u}_{k}<\infty$.
Our main object is the space 
\[ V_r(A) = \{[B] :\text{ ASD on $E$}|\, \exists g:E\to E \text{ s.t. }
              \nnorm{g(B)-A}_{2}\leq r\} \quad (r>0).\]

First we prepare a lemma concerning $\Omega^0(\ad E)$.
Here we essentially use our good/bad decomposition structure.

\begin{lemma} \label{lemma: estimate on Omega^0}
(1) For $u\in L^{2,W}_{3}(\Omega^0(\ad E))$, 
\[ \norm{u}_{L^\infty(X)} \lesssim_D \nnorm{d_A u}_{2}.\]

\noindent 
(2) For $u\in L^{2,W}_{k}(\Omega^0(\ad E))$ with $k\geq 1$, 
$\nnorm{u}_{k} \lesssim_{A} \nnorm{d_A u}_{k-1}$.
Note that the implicit constant here depends on $A$. Hence this is less effective than (1).
\end{lemma}
\begin{proof}
(1) This follows from the Sobolev embedding and 
\begin{equation} \label{eq: estimate on Omega^0}
 \norm{u}_{L^2(t,t+1)} \lesssim_D \nnorm{d_A u} \quad (\forall t\in \mathbb{R}).
\end{equation}
By the same argument as in Claim \ref{claim: linear non-degeneracy estimate},
\begin{equation} \label{eq: non-degeneracy estimate}
  \norm{u}_{L^2(n,n+1)} \lesssim_D \norm{d_A u}_{L^2(n,n+1)} \quad (\forall n\in G'(A)).
\end{equation}

Take $t\in (n_1,n_2)$ with $|t-n_1|\leq |t-n_2|$ (other cases can be treated in the same way).
For each $n_1 <s<n_1+1$
\[ |u(t,\theta)|\leq |u(s,\theta)| + \left|\int_s^t |\nabla_A u|d\tau\right|
   \leq |u(s,\theta)| + \int_{n_1}^t |\nabla_A u|d\tau.\]
\[ \int_{n_1}^t |\nabla_A u|d\tau = \int_{n_1}^t e^{-\alpha(\tau-n_1)} e^{\alpha(\tau-n_1)}|\nabla_A u| d\tau
   \leq \sqrt{\int_{n_1}^t e^{-2\alpha(\tau-n_1)}d\tau}
   \sqrt{\int_{n_1}^t e^{2\alpha(\tau-n_1)}|\nabla_A u|^2 d\tau}.\]
Since $e^{\alpha|t-G'(A)|}\lesssim W(t)$, we get 
\[ \int_{n_1}^t |\nabla_A u|d\tau \lesssim \sqrt{\int_{n_1}^t W^2 |\nabla_A u|^2 d\tau}, \]
\[ |u(t,\theta)|^2 \lesssim |u(s,\theta)|^2 + \int_{n_1}^t W^2 |\nabla_A u|^2 d\tau.\]
Integrating over $(s,\theta)\in (n_1,n_1+1)\times S^3$, 
\[ \int_{S^3} |u(t,\theta)|^2 d\vol_{S^3}(\theta) \lesssim \int_{(n_1,n_1+1)\times S^3} |u|^2 d\vol
   + \int_{(n_1,t)\times S^3}W^2|\nabla_A u|^2d\vol.\]
Using (\ref{eq: non-degeneracy estimate}) 
\[ \int_{S^3} |u(t,\theta)|^2 d\vol_{S^3}(\theta) \lesssim \nnorm{d_A u}^2.\]
The desired estimate (\ref{eq: estimate on Omega^0}) follows from this.

(2) It is enough to prove $\nnorm{u}\lesssim_A \nnorm{d_A u}$, and this follows from (\ref{eq: estimate on Omega^0}) and 
\[ \int_{\{t<n_1\}\cup \{t>n_G\}} W^2|u|^2 d\vol \lesssim \nnorm{d_A u}^2.\]
For simplicity we assume $n_G=0$ and prove 
\[ \int_{t>0} W^2 |u|^2 d\vol \lesssim \nnorm{d_A u}^2.\]
We can assume that $u$ is smooth and compactly supported. Let $t>0$.
\[ |u(t,\theta)| \leq \int_t^\infty |\nabla_A u(s,\theta)|ds = 
   \int_t^\infty W(s)^{-1}W(s)|\nabla_A u(s,\theta)|ds.\]
For $0<t<s$ we have $W(t)W(s)^{-1}\lesssim e^{\alpha (t-s)}$.
Hence 
\[ W(t)|u(t,\theta)| \lesssim \int_t^\infty e^{\alpha(t-s)}W(s)|\nabla_A u(s,\theta)|ds,\]
\begin{equation*}
   \begin{split}
   W(t)^2|u(t,\theta)|^2 &\lesssim \int_t^\infty e^{\alpha(t-s)}ds 
   \int_t^\infty e^{\alpha(t-s)}W(s)^2 |\nabla_A u(s,\theta)|^2 ds \\
   &= \frac{1}{\alpha} \int_t^\infty e^{\alpha(t-s)}W(s)^2 |\nabla_A u(s,\theta)|^2 ds.
   \end{split}
\end{equation*}
Therefore 
\begin{equation*}
  \begin{split}
    \int_0^\infty W(t)^2 |u(t,\theta)|^2 dt &\lesssim 
    \int_0^\infty \left(\int_0^s e^{\alpha(t-s)}dt\right)W(s)^2 |\nabla_A u(s,\theta)|^2 ds \\
    &\leq \frac{1}{\alpha} \int_0^\infty W(s)^2 |\nabla_A u(s,\theta)|^2 ds.
  \end{split}
\end{equation*}
Thus 
\[ \int_{t>0} W^2 |u|^2 d\vol \lesssim \int_{t>0}W^2 |\nabla_A u|^2 d\vol \leq \nnorm{d_A u}^2.\]
\end{proof}

Let $d_A^{*,W}:\Omega^1(\ad E)\to \Omega^0(\ad E)$ be the formal adjoint of 
$d_A:\Omega^0(\ad E)\to \Omega^1(\ad E)$ with respect to the weighted inner product:
For compactly supported smooth $u\in \Omega^0(\ad E)$ and $a\in \Omega^1(\ad E)$ 
\[ \int_X W^2\langle d_A u, a\rangle d\vol = \int_X W^2\langle u, d_A^{*,W}a\rangle d\vol.\]
The following lemma studies the Coulomb gauge condition.

\begin{lemma} \label{lemma: Coulomb gauge}
(1)
For $u\in L^{2,W}_{1}(\Omega^0(\ad E))$ and $a\in L^{2,W}(\Omega^1(\ad E))$ with $d_A^{*,W}a=0$
(in the distributional sense)
\[ \nnorm{d_A u} + \nnorm{a} \lesssim_D \nnorm{d_A u + a}.\]

\noindent
(2) Let $k\geq 0$. For $u\in L^{2,W}_{k+1}(\Omega^0(\ad E))$ and $a\in L^{2,W}_{k}(\Omega^1(\ad E))$ with $d_A^{*,W}a=0$
\[ \nnorm{d_A u}_{k} + \nnorm{a}_{k}\lesssim_{k,D} \nnorm{d_A u + a}_{k}.\]
\end{lemma}

\begin{proof}
(1) We can suppose that $u$ is smooth and compactly supported.
Let $N=N(D)>0$ be a sufficiently large integer which will be fixed later.
It is very important that several implicit constants below do not depend on $N$.
Recall $G'(A) = \{n_1<\dots<n_G\}$. Let $G=qN+r$ with $0\leq r<N$.
We decompose $\mathbb{R}$ as follows:
\[ \mathbb{R} = (-\infty,n_N]\cup [n_N,n_{2N}]\cup \dots \cup [n_{(q-1)N},n_{qN}]\cup [n_{qN},\infty).\]
We call these intervals $I_0, I_1,\dots,I_q$ respectively.
If $q=0$, then we simply set $I_0=\mathbb{R}$.
We set $I_{-1}=I_{q+1}=\emptyset$.
For $0\leq k\leq q$ we take a cut-off function $\varphi_k:\mathbb{R}\to [0,1]$ such that 
$\varphi_k=1$ on $I_k$, $\supp(\varphi_k) \subset I_{k-1}\cup I_k\cup I_{k+1} =:J_k$ and 
\begin{equation} \label{eq: support of cut-off}
 \supp(d\varphi_k) \subset \bigcup_{n\in G'(A)} (n,n+1), \quad |d\varphi_k|\lesssim \frac{1}{N}.
\end{equation}

From $d_A^{*,W}a=0$
\[ \int_X W^2\langle d_A u+a,d_A(\varphi_k u)\rangle d\vol 
   = \int_X W^2\langle d_A u,d_A(\varphi_k u)\rangle d\vol.\]
\begin{equation*}
   \begin{split}
    \biggl| \int_X W^2&\langle d_A u,d_A(\varphi_k u)\rangle d\vol \biggr| \gtrsim
   \int_{I_k}W^2|d_A u|^2 d\vol -\frac{1}{N} \int_{\supp(d\varphi_k)}W^2|d_A u| |u| d\vol \\
   &\geq \int_{I_k}W^2|d_A u|^2 d\vol - \frac{1}{N}\sqrt{\int_{J_k}W^2|d_A u|^2 d\vol}
   \sqrt{\int_{\supp(d\varphi_k)}W^2 |u|^2d\vol}.
   \end{split}
\end{equation*}
\begin{equation*}
   \begin{split}
   \biggl|\int_X W^2 &\langle d_A u+a,d_A(\varphi_k u)\rangle d\vol \biggr| \leq 
   \sqrt{\int_{J_k}W^2|d_A u+a|^2d\vol}\sqrt{\int_X W^2|d_A(\varphi_k u)|^2d\vol} \\
   &\lesssim \sqrt{\int_{J_k}W^2|d_A u+a|^2d\vol} \sqrt{\int_{\supp(d\varphi_k)}W^2 |u|^2d\vol 
   + \int_{J_k} W^2|d_A u|^2d\vol}.
   \end{split}
\end{equation*}
From (\ref{eq: non-degeneracy estimate}) in the proof of Lemma \ref{lemma: estimate on Omega^0} and the above 
(\ref{eq: support of cut-off}),
\[ \int_{\supp(d\varphi_k)}W^2|u|^2d\vol \lesssim 
   \int_{\supp(d\varphi_k)}|u|^2d\vol \lesssim_D \int_{J_k}|d_A u|^2 d\vol 
   \lesssim \int_{J_k} W^2 |d_A u|^2 d\vol.\]
Combining these estimates,
\begin{equation*}
   \begin{split}
    \int_{I_k}& W^2 |d_A u|^2 d\vol \\
   &\lesssim_D \sqrt{\int_{J_k}W^2 |d_A u|^2 d\vol}
   \left(\sqrt{\int_{J_k}W^2 |d_A u+a|^2d\vol} + \frac{1}{N}\sqrt{\int_{J_k}W^2|d_A u|^2 d\vol}\right).
   \end{split}
\end{equation*}
Set 
\[ R = \max_k \sqrt{\int_{I_k} W^2 |d_A u|^2 d\vol}, \quad 
   S = \max_k \sqrt{\int_{I_k}W^2 |d_A u+a|^2d\vol}.\]
Then we get 
\[ R^2 \leq C(D) \left(S+\frac{R}{N}\right)R, \text{ i.e. } R\leq C(D)S + \frac{C(D)}{N}R.\]
We choose $N=N(D)$ so that $C(D)/N<1/2$. Then 
$R\leq 2C(D)S$. 
We have $\nnorm{d_A u}\leq R$ and $S\lesssim_D \nnorm{d_A u+a}$.
Thus $\nnorm{d_A u} \lesssim_D \nnorm{d_A u + a}$.
Then $\nnorm{a} \lesssim_D \nnorm{d_A u+a}$.

(2) Let $k\geq 1$.
By the elliptic regularity of the operator $d_A^{*,W}+d_A^+$, 
\[ \nnorm{d_A u}_{k} \lesssim_{k,D}\nnorm{d_A u} + \nnorm{(d_A^{*,W} + d_A^+)d_Au}_{k-1}.\]
We have $(d_A^{*,W} + d_A^+)d_A u = d_A^{*,W}d_A u = d_A^{*,W}(d_A u + a)$.
Hence by (1)
\[ \nnorm{d_A u}_{k} \lesssim_{k,D} \nnorm{d_A u} + \nnorm{d_A u + a}_{k} 
   \lesssim_D \nnorm{d_A u+a}_{k}.\]
\end{proof}

Recall the Weitzenb\"{o}ck formula (\ref{eq: Weitzenbock formula}):
\[ d_A^+d_A^* \phi = \frac{1}{2}(\nabla_A^*\nabla_A + 2) \phi \quad (\phi\in \Omega^+(\ad E)).\]
Here $d_A^*$ and $\nabla_A^*$ are the formal adjoints of $d_A$ and $\nabla_A$
with respect to the standard (non-weighted) inner products.
For any smooth $\eta\in \Omega^+(\ad E)$ with $\norm{\eta}_{L^\infty(X)}<\infty$ there uniquely exists 
a smooth $\phi\in \Omega^+(\ad E)$ satisfying $\norm{\phi}_{L^\infty(X)}<\infty$ and 
$d_A^+d_A^*\phi = \eta$. We denote this $\phi$ by $(d_A^+d_A^*)^{-1}\eta$.
(See Section \ref{section: instanton approximation} and \cite[Appendix]{Matsuo--Tsukamoto}.)
We need to study the behavior of $(d_A^+ d_A^*)^{-1}$ under the weighted norms.

\begin{lemma}  \label{lemma: construction of P_A}
For any $k\geq 0$ and any compactly supported smooth $\eta\in \Omega^+(\ad E)$
\[ \nnorm{(d_A^+d_A^*)^{-1}\eta}_{k+2} \lesssim_{k,D} \nnorm{\eta}_{k}.\]
So we can uniquely extend the operator $(d_A^+ d_A^*)^{-1}$ to a bounded linear map from 
$L^{2,W}_{k}(\Omega^+(\ad E))$ to $L^{2,W}_{k+2}(\Omega^+(\ad E))$.
We set 
\[ P_A := d_A^*(d_A^+ d_A^*)^{-1}: L^{2,W}_{k}(\Omega^+(\ad E))\to L^{2,W}_{k+1}(\Omega^1(\ad E)).\]
This satisfies $\norm{P_A \eta}_{k+1}\lesssim_{k,D} \nnorm{\eta}_{k}$.
\end{lemma}
\begin{proof}
Set $\phi = (d_A^+ d_A^*)^{-1}\eta$. It is enough to prove 
$\nnorm{\phi}\lesssim \nnorm{\eta}$. 
By the Green kernel estimate (\ref{eq: Green kernel estimate})
\[ |\phi(x)|\lesssim \int_X g(x,y)|\eta(y)|d\vol(y).\]
We have $\int_X g(x,y)d\vol(y)\lesssim 1$ (uniformly in $x$) and 
$g(x,y)\lesssim e^{-\sqrt{2}\, \dist(x,y)}$ for $\dist(x,y)>1$.
Set $h(x,y) = W(x)W(y)^{-1}g(x,y)$. 
\[ W(x)|\phi(x)|\lesssim \int_X h(x,y)W(y)|\eta(y)|d\vol(y).\]
Since $e^{\alpha|t-G'(A)|}\lesssim W(t)\lesssim e^{\alpha|t-G'(A)|}$ 
\[ W(x)W(y)^{-1} \lesssim e^{\alpha\, \dist(x,y)}.\]
Hence (noting $\alpha<1<\sqrt{2}$)
\[ \int_X h(x,y)d\vol(y) \lesssim 1 \> (\text{uniformly in $x$}), \quad 
   h(x,y) \lesssim e^{(\alpha-\sqrt{2})\dist(x,y)} \quad (\dist(x,y)>1).\]
From the former condition 
\[ W(x)^2|\phi(x)|^2 \lesssim \int_X h(x,y)W(y)^2|\eta(y)|^2d\vol(y).\]

We denote by $t$ and $s$ the $\mathbb{R}$-coordinates of $x, y\in \mathbb{R}\times S^3$ respectively.
\begin{equation*}
    \begin{split}
   \int_{n_i<t<n_{i+1}} & W(x)^2 |\phi(x)|^2 d\vol(x) 
   \lesssim \int_X \left(\int_{n_i< t <n_{i+1}}h(x,y)d\vol(x)\right) W(y)^2|\eta(y)|^2 d\vol(y) \\
   =& \underbrace{\int_{n_i-1\leq s\leq n_{i+1}+1}\left(\int_{n_i<t<n_{i+1}}h(x,y)d\vol(x)\right)W(y)^2|\eta(y)|^2 d\vol(y)}_{(I)} \\
   &+ \underbrace{\int_{\{s< n_i-1\}\cup \{s> n_{i+1}+1\}}
    \left(\int_{n_i<t<n_{i+1}}h(x,y)d\vol(x)\right)W(y)^2|\eta(y)|^2 d\vol(y)}_{(II)}.
   \end{split}
\end{equation*}
We have $\int_{n_i<t<n_{i+1}}h(x,y)d\vol(x)\lesssim 1$. 
So the term $(I)$ is $\lesssim \nnorm{\eta}$. When $s < n_i-1$ or $s > n_{i+1}+1$, 
\[  \int_{n_i<t<n_{i+1}}h(x,y)d\vol(x) \lesssim \int_{n_i}^{n_{i+1}}e^{(\alpha-\sqrt{2})|t-s|}dt
    \lesssim \max\left(e^{(\alpha-\sqrt{2})|s-n_i|}, e^{(\alpha-\sqrt{2})|s-n_{i+1}|}\right). \]
Then the term $(II)$ is also $\lesssim \nnorm{\eta}$.
Thus we conclude $\nnorm{\phi}\lesssim \nnorm{\eta}$.
\end{proof}

We define $H^{1,W}_A$ as the space of $a\in \Omega^1(\ad E)$ satisfying 
$d_A^{*,W}a=d_A^+a=0$ and $\nnorm{a}<\infty$.
All the norms $\nnorm{\cdot}_{k,A}$ $(k\geq 0)$ are equivalent over $H^{1,W}_A$ by the elliptic regularity.

\begin{lemma} \label{lemma: index formula}
\[ \dim H^{1,W}_A = 8c_2(A)+3, \quad c_2(A) := \frac{1}{8\pi^2}\int_X |F_A|^2 d\vol.\]
\end{lemma}
\begin{proof}
We set $\mathcal{D}_A =d_A^{*,W}+d_A^+: L^{2,W}_{1}(\Omega^1(\ad E))\to L^{2,W}(\Omega^0(\ad E)\oplus \Omega^+(\ad E))$.
$H^{1,W}_A$ is the kernel of $\mathcal{D}_A$. 
We will show that $\mathcal{D}_A$ is surjective.
The map 
\[ d_A^{*,W}d_A:L^{2,W}_{2}(\Omega^0(\ad E))\to L^{2,W}(\Omega^0(\ad E)) \]
is injective and has a closed range by Lemma \ref{lemma: estimate on Omega^0} (2).
So it is an isomorphism by the principle of orthogonal projection.
(See the proof of Lemma \ref{lemma: Hodge decomposition} (2) below.)
Let $(u,\eta)\in L^{2,W}(\Omega^0(\ad E)\oplus \Omega^+(\ad E))$.
We can find $v\in L^{2,W}_{2}(\Omega^0(\ad E))$ satisfying $d_A^{*,W}d_A v=u-d_A^{*,W}P_A\eta$.
By $d_A^+P_A=1$
\[ \mathcal{D}_A(d_A v+P_A\eta) = (d_A^{*,W}d_A v + d_A^{*,W}P_A\eta, \eta) = (u,\eta).\]
Thus $\mathcal{D}_A$ is surjective. 
Therefore $\dim H^{1,W}_A = \dim \mathrm{Ker}(\mathcal{D}_A)$ is equal to the index of $\mathcal{D}_A$.
The calculation of $\mathrm{index}(\mathcal{D}_A)$ is standard, and we get 
$\mathrm{index}(\mathcal{D}_A) = 8c_2(A)+3$ by 
Donaldson \cite[Proposition 3.19]{Donaldson}.
\end{proof}

\begin{lemma} \label{lemma: Hodge decomposition}
(1) Let $k\geq 1$. For any $u\in L^{2,W}_{k+1}(\Omega^0(\ad E))$, $a\in H^{1,W}_A$ and $\eta\in L^{2,W}_{k-1}(\Omega^+(\ad E))$
\[ \nnorm{d_A u}_{k} + \nnorm{a} + \nnorm{\eta}_{k-1} \lesssim_{k,D} \nnorm{d_A u +a + P_A\eta}_{k}.\]

\noindent 
(2) Let $k\geq 1$. We define a map 
\[ \Phi: L^{2,W}_{k+1}(\Omega^0(\ad E))\oplus H^{1,W}_A\oplus L^{2,W}_{k-1}(\Omega^+(\ad E))
   \to L^{2,W}_{k}(\Omega^1(\ad E)) \]
by $\Phi(u,a,\eta) = -d_A u + a + P_A\eta$. 
Then $\Phi$ is an isomorphism.
\end{lemma}

\begin{proof}
(1) Set $b=d_A u+a+P_A\eta$. $d_A^+ b=\eta$. 
So $\nnorm{\eta}_{k-1}\lesssim_{k,D}\nnorm{b}_{k}$.
By Lemma \ref{lemma: Coulomb gauge} (2)
\begin{equation*}
  \nnorm{d_A u}_{k}+\nnorm{a} \lesssim_{k,D} \nnorm{d_A u+a}_{k} \leq \nnorm{b}_{k} + \nnorm{P_A\eta}_{k}
  \lesssim_{k,D} \nnorm{b}_{k} + \nnorm{\eta}_{k-1} \lesssim_{k,D} \nnorm{b}_{k}.  
\end{equation*}

(2) It is enough to prove that $\Phi$ is surjective. 
Take $b\in L^{2,W}_{k}(\Omega^1(\ad E))$. Set $\eta = d_A^+b$ and $b'=b-P_A\eta$. 
This satisfies $d_A^+b'=0$.
By Lemma \ref{lemma: estimate on Omega^0} (2), the space $d_A(L^{2,W}_{1}(\Omega^0(\ad E)))$ is closed 
in $L^{2,W}(\Omega^1(\ad E))$. 
So let $b'=-d_A u + a$ ($u\in L^{2,W}_{1}(\Omega^0(\ad E))$)
be the orthogonal decomposition with respect to the weighted inner product:
\[ \int_X W^2\langle d_A v, a\rangle d\vol = 0 \quad (\forall v\in L^{2,W}_{1}(\Omega^0(\ad E))).\]
Then $d_A^{*,W}a=0$. Moreover 
$d_A^+a=d_A^+(b'+d_A u) = 0$. Hence $a\in H^{1,W}_A$.
We have $d_A u = a-b'\in L^{2,W}_{k}$. So $u\in L^{2,W}_{k+1}$.
$b=-d_A u+a+P_A\eta = \Phi(u,a,\eta)$.
Thus $\Phi$ is surjective. 
\end{proof}

Let $a\in H^{1,W}_A$. The connection $A+a$ is an approximate solution of the ASD equation.
In the next lemma, we perturb it and construct a genuine solution.

\begin{lemma} \label{lemma: solving ASD equation}
We can choose $r_1=r_1(D)>0$ so that the following statements hold.

\noindent 
(1) For any $a\in H^{1,W}_A$ with $\nnorm{a}\leq r_1$ there uniquely exists $\eta\in L^{2,W}_{1}(\Omega^+(\ad E))$
satisfying 
\[  \quad F^+(A+a+P_A\eta)=0, \quad  \nnorm{\eta}_{1}\leq r_1.\]
We denote this $\eta$ by $\eta_a$ and set $\tilde{a} = a+P_A\eta_a$.

\noindent 
(2) For any $a,b\in H^{1,W}_A$ with $\nnorm{a}, \nnorm{b}\leq r_1$ 
\[ \norm{\tilde{a}-\tilde{b}}_{L^\infty(X)} \lesssim_D \nnorm{a-b}.\]
\end{lemma}

\begin{proof}
(1) $F^+(A+a+P_A\eta)=\eta+\{(a+P_A\eta)^2\}^+$.
Set $Q(\eta) = -\{(a+P_A\eta)^2\}^+$ for $\eta\in L^{2,W}_{1}(\Omega^+(\ad E))$.
If $\eta_1,\eta_2\in  L^{2,W}_{1}(\Omega^+(\ad E))$ satisfy 
$\nnorm{\eta_1}_{1}, \nnorm{\eta_2}_{1}\leq r_1$, then 
\[ \nnorm{Q(\eta_1)}_{1} \lesssim_D r_1^2,\quad  \nnorm{Q(\eta_1)-Q(\eta_2)}_{1} \lesssim_D r_1 \nnorm{\eta_1-\eta_2}_{1}.\]
Here we have used $L^{2,W}_{2}\times L^{2,W}_{2} \rightarrow L^{2,W}_{1}$.
So if we choose $r_1>0$ sufficiently small, then $Q$ becomes a contraction map 
over $\{\eta\in L^{2,W}_{1}(\Omega^+(\ad E))|\nnorm{\eta}_{1}\leq r_1\}$.
Thus the statement (1) follows.

(2) We have $\eta_a = -\{(a+P_A\eta_a)^2\}^+$ and $\eta_b = -\{(b+P_A\eta_b)^2\}^+$. Hence 
\begin{equation*}
    \nnorm{\eta_a-\eta_b}_{1} \lesssim_D r_1\left(\nnorm{a-b}+\nnorm{\eta_a-\eta_b}_{1}\right).
\end{equation*}
If $r_1$ is sufficiently small, then $\nnorm{\eta_a-\eta_b}_{1}\lesssim_D \nnorm{a-b}$.
The rest of the argument is a bootstrapping.
\end{proof}

The next lemma is a conclusion of analytic arguments in this section.
This is a non-linear version of Lemma \ref{lemma: Hodge decomposition}.

\begin{lemma} \label{lemma: continuity method}
We can choose $r_0=r_0(D)>0$ so that the following statement holds.
For any connection $B$ on $E$ with $\nnorm{B-A}_{2}\leq r_0$ there exists 
$(u,a,\eta)\in L^{2,W}_{3}(\Omega^0(\ad E))\oplus H^{1,W}_A\oplus L^{2,W}_{1}(\Omega^+(\ad E))$ satisfying 
\[ B=e^u(A+a+P_A\eta), \quad \nnorm{d_A u}_{2}+\nnorm{a}+\nnorm{\eta}_{1} < r_1.\]
Here $r_1$ is the positive constant introduced in Lemma \ref{lemma: solving ASD equation}.
\end{lemma}

\begin{proof}
Let $r_0=r_0(D)$ and $r_2=r_2(D)$ be two positive numbers which will be fixed later.
They will satisfy $0<r_0\ll r_2 < r_1$
We use a continuity method.
The crucial point is that by Lemma \ref{lemma: estimate on Omega^0} (1) 
\begin{equation} \label{eq: key estimate in continuity method}
 \norm{u}_{L^\infty(X)} \lesssim_D \nnorm{d_A u}_2 \quad (u\in L^{2,W}_3(\Omega^0(\ad E))).
\end{equation}

Set $B=A+b$ with $\nnorm{b}_{2}\leq r_0$.
We define $\mathcal{T} \subset [0,1]$ as the set of $0\leq t\leq 1$ such that there exists 
$(u_t,a_t,\eta_t)\in L^{2,W}_{3}(\Omega^0(\ad E))\oplus H^{1,W}_A\oplus L^{2,W}_{1}(\Omega^+(\ad E))$
satisfying 
\begin{equation} \label{eq: condition in continuity method}
 A+ tb = e^{u_t}(A+a_t+P_A(\eta_t)), \quad \nnorm{d_A u_t}_{2} + \nnorm{a_t}+\nnorm{\eta_t}_{1} < r_2.
\end{equation}
The origin $0$ is contained in $\mathcal{T}$. 
We will shows that $\mathcal{T}$ is closed and open.
Then $1\in \mathcal{T}$ and the proof is completed.

\textbf{Step 1}. We show that $\mathcal{T}$ is closed. Take $t\in \mathcal{T}$ and $(u_t,a_t,\eta_t)$ satisfying 
the above (\ref{eq: condition in continuity method}). We want to derive a priori bound.
\begin{equation*}
   \begin{split} 
   tb =& -d_A u_t+a_t+P_A\eta_t -(d_A e^{u_t})(e^{-u_t}-1) -d_A(e^{u_t}-1-u_t) \\
   &+ (e^{u_t}-1)(a_t+P_A\eta_t)e^{-u_t} + (a_t+P_A\eta_t)(e^{-u_t}-1).
   \end{split}
\end{equation*}
By (\ref{eq: key estimate in continuity method}) we get 
$\nnorm{-d_A u_t + a_t + P_A\eta_t}_{2} \lesssim_D r_0+r_2^2$.
By Lemma \ref{lemma: Hodge decomposition} (1), we can choose $r_0$ and $r_2$ so that 
\begin{equation} \label{eq: a priori bound}
 \nnorm{d_A u_t}_{2} + \nnorm{a_t} + \nnorm{\eta_t}_{1} \leq \frac{r_2}{2}.
\end{equation}
Then the rest of the argument is standard.
Suppose $\{t_i\}\subset \mathcal{T}$ is a sequence converging to $t_\infty\in [0,1]$.
Then by Lemma \ref{lemma: estimate on Omega^0} (2) the sequence $(u_{t_i}, a_{t_i}, \eta_{t_i})$ is bounded in 
$L^{2,W}_{3}\oplus H^{1,W}_A\oplus L^{2,W}_{1}$. So we can assume that it weakly converges to some 
$(u_{t_\infty}, a_{t_\infty},\eta_{t_\infty})$. From the above bound (\ref{eq: a priori bound}) we get 
\[ \nnorm{d_A u_{t_\infty}}_{2} + \nnorm{a_{t_\infty}} + \nnorm{\eta_{t_\infty}}_{1} \leq \frac{r_2}{2} < r_2.\]
Hence it satisfies (\ref{eq: condition in continuity method}) for $t=t_\infty$. Thus $t_\infty \in \mathcal{T}$.

\textbf{Step 2}. We show that $\mathcal{T}$ is open in $[0,1]$.
Take $t\in \mathcal{T}$. We want to show that $t$ is an inner point. Consider the map 
\begin{equation} \label{eq: implicit function theorem in continuity method}
   f:L^{2,W}_{3}(\Omega^0(\ad E))\oplus H^{1,W}_A\oplus L^{2,W}_{1}(\Omega^+(\ad E))
   \to L^{2,W}_{2}(\Omega^1(\ad E)) 
\end{equation}
defined by $f(u,a,\eta) = e^u(A+a+P_A\eta)-A$.
It is enough to prove that the derivative $(df)_{(0,a_t,\eta_t)}$ is an isomorphism.
\[ (df)_{(0,a_t,\eta_t)}(u,a,\eta) = -d_A u+a+P_A\eta - [a_t+P_A\eta_t,u].\]
Here it is convenient to consider that the left-hand-side of (\ref{eq: implicit function theorem in continuity method})
is endowed with the norm $\nnorm{d_A u}_2+\nnorm{a}+\nnorm{\eta}_1$.
By Lemma \ref{lemma: Hodge decomposition} the map $\Phi(u,a,\eta) := -d_A u+ a+P_A\eta$ is an isomorphism from 
$L^{2,W}_{3}\oplus H^{1,W}_A\oplus L^{2,W}_{1}$ to $L^{2,W}_{2}$ with
$\nnorm{d_A u}_2+\nnorm{a}+\nnorm{\eta}_1 \lesssim_D \nnorm{\Phi(u,a,\eta)}_2$.
By (\ref{eq: key estimate in continuity method}) and (\ref{eq: condition in continuity method}) 
\[ \nnorm{[a_t+P_A\eta_t,u]}_{2} \lesssim_D r_2 \nnorm{d_A u}_{2}.\]
So if $r_2$ is chosen sufficiently small, then the derivative $(df)_{(0,a_t,\eta_t)}$ is isomorphic.
\end{proof}

Then we can prove Proposition \ref{prop: quantitative deformation theory}.
Recall that for connections $B_1$ and $B_2$ on $E$ we defined 
\[ \dist_{L^\infty}([B_1],[B_2]) = \inf_{g:E\to E}\norm{g(B_1)-B_2}_{L^\infty(X)}.\]

\begin{proposition}[$=$ Proposition \ref{prop: quantitative deformation theory}]
There exists $C_0=C_0(D)>0$ such that for any $0<\varepsilon<1$
\[ \#_{\mathrm{sep}}(V_{r_0}(A), \dist_{L^\infty}, \varepsilon) \leq (C_0/\varepsilon)^{8c_2(A)+3}.\]
Here $r_0=r_0(D)$ is the positive constant introduced in Lemma \ref{lemma: continuity method}.
\end{proposition}

\begin{proof}
Set $B_{r_1}(H^{1,W}_A) = \{a\in H^{1,W}_A|\nnorm{a}\leq r_1\}$.

\begin{claim}
There exist $C_2=C_2(D)>0$ and a map $f:V_{r_0}(A)\to B_{r_1}(H^{1,W}_A)$ such that for any 
$[B_1],[B_2]\in V_{r_0}(A)$ 
\[ \dist_{L^\infty}([B_1],[B_2]) \leq C_2 \nnorm{f([B_1])-f([B_2])}.\]
\end{claim}
\begin{proof}
Take $[B]\in V_{r_0}(A)$. By Lemma \ref{lemma: continuity method} we can find 
$(a,\eta) \in H^{1,W}_A\oplus L^{2,W}_{1}(\Omega^+(\ad E))$ satisfying 
\[ [B]=[A+a+P_A\eta], \quad \nnorm{a} + \nnorm{\eta}_{1} < r_1.\]
Since $B$ is ASD, $F^+(A+a+P_A\eta)=0$. 
Then by Lemma \ref{lemma: solving ASD equation} (1) we have $\eta=\eta_a$ and 
$[B]=[A+\tilde{a}]$. We set $f([B])=a$.

Take $[B_1],[B_2]\in V_{r_0}(A)$ and set $a_1=f([B_1])$ and $a_2=f([B_2])$.
We have $[B_1]=[A+\tilde{a}_1]$ and $[B_2]=[A+\tilde{a}_2]$.
By Lemma \ref{lemma: solving ASD equation} (2) 
\[ \dist_{L^\infty}([B_1],[B_2]) \leq \norm{\tilde{a}_1-\tilde{a}_2}_{L^\infty(X)} 
   \lesssim_D \nnorm{a_1-a_2}.\]
\end{proof}

By Lemma \ref{lemma: separated set} and Example \ref{example: separated set of Banach ball}
\[ \#_{\mathrm{sep}}(V_{r_0}(A),\dist_{L^\infty},\varepsilon) \leq 
   \#_{\mathrm{sep}}(B_{r_1}(H^{1,W}_A),\nnorm{\cdot},\varepsilon/C_2)
   \leq \left(\frac{1+2r_1C_2}{\varepsilon}\right)^{\dim H^{1,W}_A}.\]
By Lemma \ref{lemma: index formula}, $\dim H^{1,W}_A = 8c_2(A)+3$.
Thus we get the conclusion.
\end{proof}

We have completed all the proofs of Theorem \ref{thm: main theorem}.

\begin{remark} \label{remark: more general case}
By the same argument we can prove the following more general result:
Let $\moduli\subset \moduli_d$ be an $\mathbb{R}$-invariant closed subset.
Then 
\[ \dim(\moduli:\mathbb{R}) \leq 8\sup_{[A]\in \moduli} \rho(A).\]
But we don't have any reasonable lower bound on the mean dimension for general $\moduli$.
\end{remark}

\vspace{10mm}

\address{ Masaki Tsukamoto \endgraf
Department of Mathematics, Kyoto University, Kyoto 606-8502, Japan}

\textit{E-mail address}: \texttt{tukamoto@math.kyoto-u.ac.jp}

\end{document}